\documentclass[11pt]{amsart}
\usepackage{amssymb, amsmath,latexsym,amsfonts,amsbsy, amsthm,mathtools,graphicx,color}
\usepackage{float}
\usepackage{hyperref}

\numberwithin{equation}{section}
\allowdisplaybreaks

\usepackage{mathrsfs,enumerate,extarrows}
\let\al=\alpha

\let\g=\gamma

\let\la=\lambda

\let\f=\frac

\let\na=\nabla

\let\pa=\partial

\let\epsilon=\varepsilon

\def\cF{{\mathcal F}}

\def\R{\mathbb R}

\def\Z{\mathbb Z}

\def\dx{\mathrm{d}x}
\def\dt{\mathrm{d}t}
\def\ds{\mathrm{d}s}

\def\dive{\mathop{\rm div}\nolimits}

\def\eqdef{\buildrel\hbox{\footnotesize def}\over =}

\newcommand{\ssubset}{\subset\joinrel\subset}

\newcommand{\andf}{~\text{ and }~}
\newcommand{\with}{~\text{ with }~}

\newcommand{\beq}{\begin{equation}}
	\newcommand{\eeq}{\end{equation}}
\newcommand{\ben}{\begin{eqnarray}}
	\newcommand{\een}{\end{eqnarray}}
\newcommand{\beno}{\begin{eqnarray*}}
	\newcommand{\eeno}{\end{eqnarray*}}

\newtheorem*{Theorem}{Theorem}

\newtheorem{theorem}{Theorem}[section]
\newtheorem{definition}[theorem]{Definition}
\newtheorem{lemma}[theorem]{Lemma}

\newtheorem{corollary}[theorem]{Corollary}
\theoremstyle{remark}
\newtheorem{remark}{Remark}[section]

\begin{document}

\title[Global well-posedness of inhomogeneous Navier-Stokes equations]
{Global well-posedness of inhomogeneous Navier-Stokes equations with bounded density}

\author[T. Hao]{Tiantian Hao}
\address[T. Hao]{
	School of Mathematical Sciences, Peking University, Beijing 100871, China}
\email{haotiantian@pku.edu.cn}

\author[F. Shao]{Feng Shao}
\address[F. Shao]{School of Mathematical Sciences, Peking University, Beijing 100871,  China}
\email{fshao@stu.pku.edu.cn}

\author[D. Wei]{Dongyi Wei}
\address[D. Wei]{
	School of Mathematical Sciences, Peking University, Beijing 100871, China}
\email{jnwdyi@pku.edu.cn}

\author[Z. Zhang]{Zhifei Zhang}
\address[Z. Zhang]{
	School of Mathematical Sciences, Peking University, Beijing 100871, China}
\email{zfzhang@math.pku.edu.cn}

\date{\today}

\begin{abstract}
 In  this paper, we solve Lions' open problem: {\it the uniqueness of weak solutions for the 2-D inhomogeneous Navier-Stokes equations (INS)}. We first prove the global existence of  weak solutions to 2-D (INS) with bounded initial density and initial velocity in $L^2(\mathbb R^2)$. Moreover, if the initial density is bounded away from zero, then our weak solution equals to Lions' weak solution, which in particular implies the uniqueness of Lions' weak solution. We also extend a celebrated result by Fujita and Kato on the 3-D incompressible Navier-Stokes equations to 3-D (INS): {\it  the global well-posedness of 3-D (INS) with bounded initial density and initial velocity being small in $\dot H^{1/2}(\mathbb R^3)$}. 
 The proof of the uniqueness is based on a surprising finding that the estimate $t^{1/2}\nabla u\in L^2(0,T; L^\infty(\mathbb R^d))$ instead of $\nabla u\in L^1(0, T; L^\infty(\mathbb R^d))$ is enough to ensure the uniqueness of the solution.
 \end{abstract}

\maketitle

\section{Introduction}
In this paper, we consider the inhomogeneous Navier-Stokes equations in $ \mathbb{R}^{+}\times\R^d$ ($d=2,3$):
\begin{equation}\label{INS}
	\left\{
	\begin{array}{l}
		\partial_t\rho+u\cdot\nabla \rho=0,\\
		\rho(\partial_tu+u\cdot\nabla u)-\Delta u+\nabla P=0,\\
		\dive u = 0,\\
		(\rho, u)|_{t=0} =(\rho_{0}, u_0),
	\end{array}
	\right.
\end{equation}
where $\rho,~u$ stand for the density and velocity of the fluid respectively, and $P$ is a scalar pressure function.  This system is known as a model for the evolution of a multi-phase flow consisting of several immiscible, incompressible fluids with different densities. We refer to \cite{Lions_vol1} for a detailed derivation of \eqref{INS}.

Lady\v{z}enskaja and Solonnikov \cite{LS} first addressed the question of unique solvability of \eqref{INS} in a bounded domain $\Omega$. Under the assumption that $u_0\in W^{2-\frac{2}{p},p}(\Omega)(p>d)$ is divergence free and vanishes on $\pa\Omega$ and  $\rho_0\in C^1(\Omega)$ is bounded away from zero, they proved the global well-posedness in dimension $d=2$, and local well-posedness in dimension $d=3$. To obtain well-posedness for less regular initial data or global results, the following three major features of \eqref{INS} are very crucial:
\begin{itemize}
    \item the incompressible condition $\dive u=0$ implies
    \[\operatorname{meas}\{x\in\R^d: \alpha\leq\rho(t,x)\leq\beta\} \text{ is independent of }t\geq0,\quad\forall\ 0\leq\alpha\leq\beta,\]
    and in particular  $\|\rho(t)\|_{L^\infty}=\|\rho_0\|_{L^\infty}$;
    \item the conservation of energy
    \begin{equation}\label{Eq.Energy}
        \frac12\int_{\R^d}\rho(t,x)|u(t,x)|^2\,\dx+\int_0^t\|\nabla u(s,\cdot)\|_{L^2(\R^d)}^2\,\ds=\frac12\int_{\R^d}\rho_0(x)|u_0(x)|^2\,\dx;
    \end{equation}
    
    \item the scaling invariance property which states that if $(\rho,u,P)$ is a solution of \eqref{INS} on $[0, T]\times\R^d$, then for all $\la>0$, the triplet $(\rho_\la, u_\la, P_\la)$ defined by
    \begin{equation}\label{Eq.scaling}
        (\rho_\la, u_\la, P_\la)(t,x):=(\rho(\la^2t,\la x), \la u(\la^2t,\la x), \la^2P(\la^2t,\la x))
    \end{equation}
    solves \eqref{INS} on $[0, T/\la^2]\times\R^d$.
\end{itemize}

A large number of works have been devoted to proving the well-posedness of \eqref{INS} in the so-called {\it critical functional framework}, which is to say, in functional spaces with scaling invariant norms. 

When the density $\rho$ is a constant (let's say, $\rho\equiv1$), in which case \eqref{INS} becomes the classical incompressible Navier-Stokes equations (NS). The pioneering work of Leray \cite{Leray1934} proved that if the initial velocity $u_0\in L^2(\R^3)$ with $\text{div} u_0=0$, then there exists a global weak (turbulent) solution to (NS) satisfying \eqref{Eq.Energy} with an inequality. Later on, this result was extended to any domain of $\R^d$ with $d=2,3$, see \cite{CDGG}. In dimension $d=2$, the proof of uniqueness is contained in Leray \cite{Leray1933, Leray1934, Leray1934JMPA}, Lady\v{z}enskaja \cite{Ladyzhenskaya}, and Lions and Prodi \cite{LP}. However, the uniqueness of Leray's weak solution  in dimension $d=3$ remains a longstanding open problem. See \cite{ABC2022, BV2019, JS2015} for recent breakthrough on non-uniqueness of weak solutions.

 We first recall the celebrated result by Fujita and Kato \cite{Fujita-Kao1964}, 
which proves the global solvability of the 3-D incompressible Navier-Stokes equations in a critical framework.

\begin{Theorem}[\cite{Fujita-Kao1964}]\label{Thm.Fujita-Kato}
    Let $d=3$. Given a divergence-free vector field $u_0\in \dot H^{1/2}(\R^3)$ with $\|u_0\|_{\dot H^{1/2}}$ small enough, then (NS) has a unique global solution $$u\in C([0, +\infty);\dot H^{1/2}(\R^3))\cap L^4(\R^+;\dot H^1(\R^3))\cap L^2(\R^+;\dot H^{3/2}(\R^3)).$$
\end{Theorem}

Recently, several works are devoted to the extension of Leray's weak solution to 2-D (INS) and the extension of Fujita-Kato's solution to 3-D (INS), which are both in the critical functional framework.

In this framework for (INS), Danchin \cite{Danchin2003} and Abidi \cite{Abidi2007} proved that if $\rho_0$ is close to a positive constant in $\dot B^{d/p}_{p,1}(\R^d)$ and $u_0$ is sufficiently small in $\dot B^{-1+d/p}_{p,1}(\R^d)$, then there is a global solution to \eqref{INS} with the initial data $(\rho_0, u_0)$ for all $p\in(1,2d)$ and the uniqueness holds for $p\in(1,d]$. The existence result has been extended to general Besov spaces in \cite{Abidi-Paicu2007, Danchin_Mucha2012, Paicu_zhang2012} even without the size restriction for the density, see \cite{Abidi_Gui2021, Abidi_Gui_Zhang2012, Abidi_Gui_Zhang2013, Abidi_Gui_Zhang2023}, and  in \cite{Danchin_Mucha2012}, Danchin and Mucha proved the existence and uniqueness for $p\in[1,2d)$, where $\rho_0$ is close to a positive constant in a multiplier space $\mathcal M(\dot B^{-1+d/p}_{p,1})$. 

In all these aforementioned results, the density has to be at least continuous or near a positive constant, which excludes some physical cases when the density  is discontinuous and has large variations along a hypersurface (but still bounded), for example, $\rho=a\mathbf{1}_D+b$ for some $a>b\geq0$ and some bounded open set $D\subset\R^d$. 
Now we consider the general case when $0\leq \rho_0\in L^\infty$. In 1974, Kazhikhov \cite{K} proved the existence of global weak solutions if $\rho_0$ is bounded away from vacuum (i.e., $\inf \rho_0>0$) and $u_0\in H^1(\R^d)$. The no vacuum assumption was later removed by Simon \cite{S}. Then Lions \cite{Lions_vol1} extended the previous results to the case of a density dependent viscosity, where the density equation of \eqref{INS} is satisfied in a renormalized meaning. (We remark that the uniqueness of Lions' weak solution, even in dimension $d=2$, in which case Leray's weak solution for the classical (NS) is unique, is a longstanding open problem, and we will come back to this topic later in Remark \ref{Rmk.uniqueness}.) While with $u_0\in H^2(\R^d)$, Danchin and Mucha \cite{DM3} proved that the system \eqref{INS} has a unique local in time solution. Paicu, Zhang and the fourth author \cite{Paicu_Zhang_Zhang_CPDE}  proved the global well-posedness of \eqref{INS} with $\rho_0\in L^{\infty}(\R^2)$ bounded away from zero and  $u_0\in H^s(\R^2)$ for any $s>0$, and for $d=3$ they require $u_0\in H^1(\R^3)$ with $\|u_0\|_{L^2}\|\na u_0\|_{L^2}$ small. The 3-D result in \cite{Paicu_Zhang_Zhang_CPDE} was further improved to $u_0\in H^s(\R^3)$ for any $s>1/2$ in \cite{CZZ}. We also mention that in \cite{Hao_Shao_Wei_Zhang}, the authors solved the so-called {\it density patch problem} in the 2-D case, which states that if $\rho_0=\mathbf 1_D$ for a $C^{1,\alpha}$ domain $D\subset \R^2$ then the regularity of the boundary is preserved for Lions' weak solution, by showing the weak-strong uniqueness. See also \cite{DM, DZ, GG, LZ-ARMA, LZ, PT} for related results. In these works, (except for Lions' 2-D weak solutions and Leray's 2-D weak solutions,) we note that the norms of the initial velocity are not critical in the sense that the norms are not invariant under the scaling transformation \eqref{Eq.scaling}. 

The first result where $\rho_0$ is merely bounded and $u_0$ lies in a critical space was obtained by Zhang \cite{Zhang_Adv}, where he established the global existence of solutions to 3-D (INS) with initial density $0\leq\rho_0\in L^\infty(\R^3)$, bounded away from zero, and initial velocity $u_0$ sufficiently small in a critical Besov space $\dot B^{1/2}_{2,1}(\R^3)$. The uniqueness has been proved later by Danchin and Wang \cite{DW}. In \cite{DW}, Danchin and Wang proved global existence and uniqueness for $u_0\in \dot B^{-1+2/p}_{p,1}(\R^2)$ and $\|\rho_0-1\|_{L^\infty(\R^2)}$ small, where $p\in(1,2)$; in the 3-D case, they proved global existence if $u_0$ is small in $\dot B^{-1+3/p}_{p,1}(\R^3)$ for $p\in(1,3)$ and $\|\rho_0-1\|_{L^\infty(\R^3)}$ is small, and they obtained the uniqueness for $p\in(1, 2]$ and also for $p\in(2,3)$ if additionally $u_0\in L^2(\R^3)$. See also \cite{Qian_Chen_Zhang2023} for another extension of Zhang's result to asymmetric fluids. In these results, the summability index 1 in the Besov spaces insures that the gradient of the velocity is in $L^1(0,T; L^\infty)$, which is a key ingredient in the proof of uniqueness. 
In a very recent paper by Danchin \cite{Danchin2024}, for the 2-D case he proved the global existence of solutions to \eqref{INS} if $\rho_0\in L^\infty(\R^2)$ is bounded away from zero and $u_0\in L^2(\R^2)$, and he obtained the uniqueness if $u_0$ lies in a suitable subspace of $L^2(\R^2)$ resembling $\dot B^0_{2,1}(\R^2)$. We emphasize that, all these works require $\rho_0$ to be bounded away from zero for their existence, and the uniqueness holds only for $u_0$ belonging to a subspace of $\dot H^{d/2-1}(\R^d)$ or some $L^p$ type Besov space such that $\nabla u\in L^1(0,T; L^\infty)$.

\smallskip

Now we state our main results of this paper. \smallskip

In the 2-D case, we prove the global existence of weak solution for $0\leq\rho_0\in L^\infty(\R^2)$ and $u_0\in L^2(\R^2)$, and we also obtain the uniqueness if $\rho_0$ is also bounded away from zero. 

\begin{theorem}\label{Thm.existence_2D}
    Let $d=2$. Given the initial data $(\rho_0,u_0)$ satisfying $0\leq \rho_0(x)\leq \|\rho_0\|_{L^{\infty}}$ and $\rho_0\not\equiv 0,~ u_0\in L^2(\R^2),
	~\dive u_0=0$,  the system \eqref{INS} has a global weak solution $(\rho,u,\nabla P)$ with $0\leq \rho(t,x)\leq \|\rho_0\|_{L^{\infty}}$, $\sqrt\rho u\in C([0, +\infty); L^2(\R^2))$ and the following properties
    \begin{itemize}
        \item $\sqrt\rho u\in L^\infty(\R^+; L^2(\R^2))$ and $\nabla u\in L^2(\R^+; L^2(\R^2))$;
        \item $t^{1/2}\nabla u\in L^\infty(\R^+; L^2(\R^2))$ and $t^{1/2}\sqrt\rho \dot u\in L^2(\R^+; L^2(\R^2))$;
        \item $t\sqrt\rho\dot u\in L^\infty(\R^+; L^2(\R^2))$ and $t\nabla\dot u\in L^2(\R^+; L^2(\R^2))$, {where $\dot u\eqdef u_t+u\cdot\nabla u$};
        \item $t^{1/2}\nabla u\in L^2(\R^+; L^\infty(\R^2))$.
    \end{itemize}
    Moreover, if $\rho_0$ is bounded away from zero, then the  solution is unique.
\end{theorem}

In the 3-D case, we prove the global existence and uniqueness of Fujita-Kato's solution to the inhomogeneous  Navier-Stokes equations. Here we only require $0\leq \rho_0\in L^\infty(\R^3)$ (even for the uniqueness)  and $u_0\in \dot H^{1/2}(\R^3)$ small.
\begin{theorem}\label{Thm.existence}
	Let $d=3$. Given the initial data $(\rho_0,u_0)$ satisfying $0\leq \rho_0(x)\leq \|\rho_0\|_{L^{\infty}}$ and $\rho_0\not\equiv 0,~ u_0\in \dot H^{1/2}(\R^3),
	~\dive u_0=0$, there exists $\varepsilon_0>0$ depending only on $\|\rho_0\|_{L^{\infty}}$ such that if $\|u_0\|_{\dot H^{1/2}(\R^3)}< \varepsilon_0$, then the system \eqref{INS} has a unique global weak solution $(\rho,u,\nabla P)$ with $0\leq \rho(t,x)\leq \|\rho_0\|_{L^{\infty}}$, $u\in L^4(\R^+; \dot H^1(\R^3))$, $\sqrt\rho u\in C([0, +\infty); L^3(\R^3))$ and the following properties
    \begin{itemize}
        \item $\sqrt\rho u\in L^\infty(\R^+; L^3(\R^3))$, $t^{-1/4}\nabla u\in L^2(\R^+; L^2(\R^3))$, and $t^{1/4}\nabla u\in L^\infty(\R^+; L^2(\R^3))$;
        \item $t^{1/4}\sqrt\rho\dot u,\ t^{1/4}\nabla^2u,\ t^{3/4}\nabla\dot u\in L^2(\R^+; L^2(\R^3))$, and $t^{3/4}\sqrt\rho\dot u\in L^\infty(\R^+; L^2(\R^3))$, where $\dot u\eqdef u_t+u\cdot\nabla u$;
        \item $t^{1/4}\sqrt\rho u_t,\ t^{1/4}\sqrt{\rho_0} u_t,\ t^{3/4}\nabla u_t\in L^2(\R^+; L^2(\R^3))$ and $t^{3/4}\sqrt\rho u_t\in L^\infty(\R^+; L^2(\R^3))$;
        \item $t^{1/2}\nabla u\in L^2(\R^+; L^\infty(\R^3))$.
    \end{itemize}
\end{theorem}

\begin{remark}
    \begin{enumerate}
 \item Compared with Zhang's result \cite{Zhang_Adv}, we allow for vacuum density and extend the condition $u_0\in \dot B^{1/2}_{2,1}$ to $\dot B^{1/2}_{2,2}=\dot H^{1/2}$, which is exactly the functional space used by Fujita and Kato in Theorem \ref{Thm.Fujita-Kato}. Moreover, we also obtain the uniqueness  of  the solution even though here we can not get the $L^1(0, T; L^\infty(\R^3))$ estimate for $\nabla u$. Indeed, we only have $\|t^{1/2}\nabla u\|_{L^2(\R^+; L^\infty(\R^3))}\leq C\|u_0\|_{\dot H^{1/2}(\R^3)}$.
 
        \item If $\rho_0$ is bounded away from zero, then we have $u\in C([0, +\infty); \dot H^{1/2}(\R^3))$.
        \item In the proof of uniqueness, we only use the following properties of the solution: 
        \begin{align*}
            &\nabla u\in L^4(\R^+; L^2(\R^3)),\quad \sqrt\rho u\in L^\infty(\R^+; L^3(\R^3)), \quad t^{1/4}\nabla u\in L^\infty(\R^+; L^2(\R^3)),\\
            t^{3/4}&\nabla\dot u\in L^2(\R^+;L^2(\R^3)),\quad t^{1/4}\sqrt\rho u_t\in L^2(\R^+; L^2(\R^3)), \quad t^{3/4}\nabla u_t\in L^2(\R^+;L^2(\R^3)),
        \end{align*}
        and $t^{1/2}\nabla u\in L^2(\R^+; L^\infty(\R^3))$.
    \end{enumerate}
\end{remark}

The uniqueness parts of Theorem \ref{Thm.existence_2D} and Theorem \ref{Thm.existence} are consequences of the following much more general result.
\begin{theorem}\label{Thm.uniqueness}
Let $T>0$. Let $(\rho,u,\nabla P)$ and $(\bar\rho,\bar u, \nabla \bar P)$ be two solutions of \eqref{INS} on $[0,T]\times\R^d$ corresponding to the same initial data. Assume in addition that
\begin{itemize}
    \item $\sqrt{\rho}(\bar u-u)\in L^{\infty}(0,T;L^2(\R^d))$;
    \item $\nabla \bar u-\nabla u \in L^2(0,T;L^2(\R^d))$; 
    \item $t^{1/2}\nabla \bar u\in L^2(0,T;L^{\infty}(\R^d))$;
    \item Case $d=2$: $\rho_0$ is bounded away from zero and $$\nabla \bar u\in L^2(0,T;L^2(\R^2)), \quad t\dot{\bar u}\in L^{\infty}(0,T;L^2(\R^2)), \quad t\nabla \dot{\bar u}\in L^2(0,T;L^2(\R^2));$$
    \item Case $d=3$: $\nabla \bar u\in L^4(0,T;L^2(\R^3))$ and $t^{3/4}\nabla \dot{\bar u}\in L^2(0,T;L^2(\R^3))$, where $\dot{\bar{u}}:=\partial_t\bar{u}+\bar{u}\cdot\nabla \bar{u}$.
\end{itemize}
Then $(\rho,u,\nabla P)=(\bar\rho,\bar u, \nabla\bar P)$ on $[0,T]\times \R^d$.
\end{theorem}
\begin{remark}
    Note that regularity requirements in the above uniqueness result are all at the critical level in the sense of \eqref{Eq.scaling}. Here we use $t^{1/2}\nabla \bar u\in L^2(0,T;L^{\infty}(\R^d))$ instead of $\nabla \bar u\in L^1(0,T;L^{\infty}(\R^d))$, which is a {great} improvement compared with previous works on the uniqueness.
\end{remark}

\begin{remark}\label{Rmk.uniqueness}
  Theorem \ref{Thm.uniqueness} solves Lions' open problem on the uniqueness of weak solutions for 2-D (INS), see page 31 of \cite{Lions_vol1}. In section 2.5 of \cite{Lions_vol1}, Lions indicated that the uniqueness of weak solutions could be showed by constructing a more regular ``strong" solution and proving the coincidence of weak solution and the strong one. In \cite{PT}, Prange and Tan proved the uniqueness of Lions' weak solution for $u_0$ satisfying $\sqrt{\rho_0}u_0\in L^2(\R^d)$ and $\nabla u_0\in L^2(\R^d)$ and for $\rho_0$ satisfying several cases.\footnote{In \cite{PT}, if $\rho_0$ has a far-field vacuum, they need an extra compatibility condition and a technical extra condition on the weak solution.} Compared with \cite{PT}, the authors in \cite{Hao_Shao_Wei_Zhang} proved that if $u_0\in H^1(\R^2)$ and $\rho_0$ allows for a vacuum bubble or a far-field vacuum without the compatibility condition, then Lions' weak solution is unique. Note that, for dimension $d=2$, these two works both require $\nabla u_0\in L^2(\R^2)$,\footnote{This regularity, higher than the critical one, ensures the weak-strong uniqueness for those $\rho_0$ possessing no positive lower bound. We also remark that in \cite{Hao_Shao_Wei_Zhang}, the regularity $u_0\in H^1(\R^2)$ is not optimal, it is possible to prove the weak-strong uniqueness if $u_0\in H^s(\R^2)$ for some $s\in(0,1)$.} although they allow the vacuum for $\rho_0$.  Theorem \ref{Thm.uniqueness} implies that in the 2-D case, if $\rho_0\in L^\infty(\R^2)$ is bounded away from zero and $u_0\in L^2(\R^2)$, then our solution is the same as Lions' weak solution \cite{Lions_vol1}. In particular, we prove that Lions' weak solution is unique if $0\leq \rho_0\in L^\infty(\R^2)$ is bounded away from zero and $u_0\in L^2(\R^2)$.
\end{remark}

\noindent\textbf{Notations.}

\begin{itemize}
    \item $\R^+:=(0, +\infty)$. $D_t\eqdef{(\pa_t+u\cdot \nabla)}$ is the material derivative.
    \item For a Banach space $X$ and an interval $I\subset\R$, we denote by $C(I;X)$ the set of continuous functions on $I$ with values in $X$. For $p\in[1,+\infty]$, the notation $L^p(I; X)$ stands for the collection of measurable functions on $I$ with values in $X$, such that $t\mapsto\|f(t)\|_X$ belongs to $L^p(I)$. For any $T>0$, we abbreviate $L^p((0, T); X)$ to $L^p(0, T; X)$ and sometimes we further abbreviate to $L^p(0, T)$ if there is no confusion. 
    \item $B(a,R):=\{x\in \R^d:|x-a|<R\},\, B_R:=B(0,R), ~\forall\ a \in \R^d,~ R>0$.
    \item For $s\in\R$ and $p\in[1,+\infty]$, we denote by $\dot W^{s,p}(\R^d)$ (or shortly $\dot W^{s,p}$) the standard homogeneous Sobolev spaces, and we also denote $\dot H^s:=\dot W^{s,2}$.
    \item We shall always denote $C$ to be a positive absolute constant which may vary from line to line. The dependence of the constant $C$ will be explicitly indicated if there are any exceptions.
\end{itemize}

\section{A priori estimates in the 2-D case}

This part is devoted to the proof of a priori estimates for \eqref{INS} in the $2$-D case. All estimates in this section are essentially taken from \cite{Hao_Shao_Wei_Zhang}. Note that in \cite{Hao_Shao_Wei_Zhang}, some estimates hold only if additionally $\nabla u_0\in L^2(\R^2)$; nevertheless, here we only use those estimates in \cite{Hao_Shao_Wei_Zhang} that depend only on the condition $u_0\in L^2(\R^2)$ and $0< \rho_0\in L^\infty(\R^2)$.

Let $(\rho,u)$ be a smooth solution to the system \eqref{INS} on $[0,+\infty)\times \R^2$ satisfying $\rho_*< \rho\leq \|\rho_0\|_{L^{\infty}}$ for some constant $\rho_*>0$.\footnote{In this and the next section, we require the density to be bounded away from zero to ensure the smoothness and appropriate decay of the solution $(\rho, u)$, which will be useful in some steps involving integration by parts. Nevertheless, in all estimates, the constant $C>0$ is independent of this positive lower bound $\rho_*$. \label{footnote_low-bound}} Firstly, standard energy estimate gives 
\begin{align}\label{energy_2}
\|\sqrt{\rho}u\|_{L^\infty(\R^+; L^2)}^2
+2\|\nabla u\|^2_{L^2(\R^+; L^2)}
\leq 2\|\sqrt{\rho_0}u_0\|_{L^2}^2.
\end{align} 

The following lemma was essentially proved in Lemma 3.2 of \cite{LSZ}  and Lemma 3.2 of \cite{Hao_Shao_Wei_Zhang}.

\begin{lemma}\label{Lem.high_order_est_2D}
There exists a positive constant 
$C$ depending only on $\|\rho_0\|_{L^{\infty}}$ and $\|\sqrt{\rho_0}u_0\|_{L^2}$ such that 
\begin{align}\label{eq1}
    \|t^{1/2}&\nabla u\|_{L^\infty(\R^+; L^2(\R^2))}+\|t^{1/2}(\sqrt{\rho}\dot{u},\nabla^2u,\nabla P)\|_{L^2(\R^+; L^2(\R^2))}\leq C,\\
   \label{eq2} &\|t(\sqrt{\rho}\dot{u},\nabla^2u,\nabla P)\|_{L^\infty(\R^+; L^2(\R^2))}+\|t\nabla \dot u\|_{L^2(\R^+; L^2(\R^2))}\leq C.
\end{align}
\end{lemma}
\begin{proof}
It was proved in \cite{LSZ} (3.11), (3.12) that (also using $0\leq \rho\leq \|\rho_0\|_{L^{\infty}}$)
\begin{align}\label{3.11}
    &\|(\nabla^2u,\nabla P)\|_{L^2}\leq C\|{\rho}\dot{u}\|_{L^2}\leq C\|\sqrt{\rho}\dot{u}\|_{L^2},\\ 
    \label{3.12}&\frac{\mathrm d}{\dt}\|\nabla u\|_{L^2}^2+\|\sqrt{\rho}\dot{u}\|_{L^2}^2\leq  C\|\nabla u\|_{L^2}^4,
\end{align}
where $C$ depends only on $\|\rho_0\|_{L^{\infty}}$. Multiplying \eqref{3.12} by $t$, using \eqref{energy_2}, Gr\"onwall's inequality and \eqref{3.11} yields \eqref{eq1}. It was proved in \cite{Hao_Shao_Wei_Zhang} (3.5), (3.7), (3.8) that\begin{align}
    &\Psi{'}(t)+\|\nabla \dot{u}\|^2_{L^2}{/2}
\leq C \|\nabla u\|^4_{L^4}+C \|\nabla P\|^2_{L^2}\|\nabla u\|^2_{L^2},\label{psi}\\
\label{psi_2}&\f14\int_{\R^2}\rho|\dot{u}|^2\,\dx-C\|\nabla u\|^4_{L^2}\leq
\Psi(t)\leq \int_{\R^2}\rho|\dot{u}|^2\,\dx+C\|\nabla u\|^4_{L^2},\\
\label{u^4}&\|\nabla u\|^4_{L^4}+\|\nabla P\|^2_{L^2}\|\nabla u\|^2_{L^2}\leq C\|\sqrt{\rho}\dot{u}\|^2_{L^2} \|\nabla u\|^2_{L^2},
\end{align}where $C$ depends only on $\|\rho_0\|_{L^{\infty}}$, and $$
\Psi(t)\eqdef \f12\int_{\R^2}\rho|\dot{u}|^2\,\dx
-\int_{\R^2} P \pa_iu^j\pa_ju^i\,\dx.$$
Multiplying \eqref{psi} by $t^2$ and using \eqref{psi_2}, \eqref{u^4},  \eqref{eq1}, \eqref{energy_2} and \eqref{3.11} yields \eqref{eq2}.
\end{proof}

In order to prove the existence of weak solution, we need a $L^\infty$ bound of $|u(t,x)|$. 
As $\rho_0\not\equiv 0$, there exists $R_0\in (0,\infty)$ such that
\begin{align}\label{c0}
\int_{B(0,R_0)} \rho_0(x)\,\dx\geq c_0>0.    
\end{align}

\if0\begin{lemma}[Lemma 2.2 of \cite{Hao_Shao_Wei_Zhang}]\label{lem2}
Assume \eqref{c0}, $t>0$ and $R\geq R_0+2t\|\sqrt{\rho_0}u_0\|_{L^2}/\sqrt{c_0}$. Then we have 
\begin{align*}
\int_{B(0,R)} \rho(t,x)\,\dx\geq c_0/4.    
\end{align*}
\end{lemma}\fi 

\begin{lemma}\label{u_L^infty}
Assume \eqref{c0}. Then there exists a constant $C>0$ depending only on $\|\rho_0\|_{L^{\infty}}$,
$\|\sqrt{\rho_0}u_0\|_{L^2}$, $R_0$ and $c_0$ such that 
\begin{align*}
\sqrt{t}|u(t,x)|\leq C[t+\ln(|x|+t+t^{-1})]^{\f12}, \quad \forall\ t>0,\ \ x\in\R^2.
\end{align*}
\end{lemma}

\begin{proof}
Let $R(t):= R_0+2t\|\sqrt{\rho_0}u_0\|_{L^2}/\sqrt{c_0}$. Then by Lemma 2.2 of \cite{Hao_Shao_Wei_Zhang}, we have 
\begin{align*}
\int_{B(0,R(t))} \rho(t,x)\,\dx\geq c_0/4.
\end{align*}
So, we obtain
\begin{align*}
\frac{c_0}{4} \inf_{x\in B(0,R(t))} |u(t,x)|^2
\leq \int_{B(0,R(t))} \rho|u|^2\,\dx\leq \|\sqrt{\rho_0}u_0\|^2_{L^2}\leq C,
\end{align*}
which implies that there exists a $x_t\in B(0,R(t))$ such that $|u(t,x_t)|\leq C_1$.

By Lemma \ref{Lem.high_order_est_2D}, we have
\begin{align*}
\|\nabla u(t)\|^2_{L^2}=\int_{\R^2}|\hat{u}(t, \xi)|^2|\xi|^2\,\mathrm d\xi\leq Ct^{-1},\quad
\|\nabla^2 u(t)\|^2_{L^2}=\int_{\R^2}|\hat{u}(t, \xi)|^2|\xi|^4\,\mathrm d\xi\leq Ct^{-2},
\end{align*}
which ensures that 
\begin{align}\label{ut}
\int_{\R^2}|\hat{u}(t,\xi)|^2|\xi|^2(1+t|\xi|^2)\,\mathrm d\xi\leq Ct^{-1}.    
\end{align}
Thus,  for any $x\neq y\in \R^2$, we have (as $|e^{i\xi\cdot x}-e^{i\xi\cdot y}|\leq\min ( |\xi||x-y|,2) $)
\begin{align}
    |u(t, x)-u(t, y)|
&\leq \int_{\R^2}|\hat{u}(t, \xi)||e^{i\xi\cdot x}-e^{i\xi\cdot y}|\,\mathrm d\xi\nonumber\\
&\leq \biggl(\int_{\R^2}|\hat{u}(t, \xi)|^2|\xi|^2(1+t|\xi|^2)\,\mathrm d\xi\biggr)^{1/2}
\biggl(\int_{\R^2}\frac{\min ( |\xi|^2|x-y|^2,4)}{|\xi|^2(1+t|\xi|^2)}\,\mathrm d\xi\biggr)^{1/2},\label{u_x_y}
\end{align}
where 
\begin{align*}
\int_{\R^2}\frac{\min ( |\xi|^2|x-y|^2,4)}{|\xi|^2(1+t|\xi|^2)}\,\mathrm d\xi&=\biggl(\int_{|\xi|\leq |x-y|^{-1}}+\int_{|\xi|> |x-y|^{-1}}\biggr)
\frac{\min ( |\xi|^2|x-y|^2,4)}{|\xi|^2(1+t|\xi|^2)}\,\mathrm d\xi\\
&\leq C\ln \left(2+|x-y|/\sqrt{t}\right),
\end{align*}
 which along with \eqref{ut} and \eqref{u_x_y} implies 
\begin{equation}\label{Eq.u(x)-u(y)}
|u(t, x)-u(t, y)|\leq Ct^{-\f12}\left[\ln \left(2+|x-y|/\sqrt{t}\right)\right]^{\f12}.
\end{equation}

Therefore, by \eqref{Eq.u(x)-u(y)}, $x_t\in B(0,R(t))$, $|u(t,x_t)|\leq C_1$ and $R(t)= R_0+2t\|\sqrt{\rho_0}u_0\|_{L^2}/\sqrt{c_0}$, we arrive at 
\begin{equation*}\begin{aligned}
|u(t,x)|&\leq |u(t,x_t)|+|u(t,x)-u(t,x_t)|
\leq C+Ct^{-\f12}[\ln \bigl(2+t^{-1}+|x|+R(t)\bigr)]^{\f12}\\
&\leq C+Ct^{-\f12}[\ln (t+|x|+t^{-1})]^\f12.
\end{aligned}\end{equation*}
Then we get
\begin{align*}
|\sqrt{t}u(t,x)|&\leq C\sqrt{t}+C[\ln (t+|x|+t^{-1})]^\f12\leq  C[t+\ln(|x|+t+t^{-1})]^{\f12}, \quad \forall\ t>0,\ \ x\in\R^2.    
\end{align*}

This completes the proof of the lemma.
\end{proof}

\begin{lemma}\label{Lem.L^2_L^infty_2D}
    There exists a constant $C>0$ depending only on $\|\rho_0\|_{L^{\infty}}$ and $\|\sqrt{\rho_0}u_0\|_{L^2}$ such that 
    \begin{equation}
        \|t^{1/2}\nabla u\|_{L^2(\R^+; L^\infty(\R^2))}\leq C.
    \end{equation}
\end{lemma}
\begin{proof}
    Following the proof of Lemma 3.5 in \cite{Hao_Shao_Wei_Zhang}, we have
    \begin{equation}
        \|\nabla u\|_{L^{\infty}(\R^2)}\leq C\|\nabla u\|_{L^2(\R^2)}^{1/2}\|\nabla \dot{u}\|_{L^2(\R^2)}^{1/2}+C\|\sqrt{\rho}\dot{u}\|_{L^2(\R^2)},
    \end{equation}
    where $C>0$ is an absolute constant. Then by Lemma \ref{Lem.high_order_est_2D}, we obtain
    \begin{align*}
     &\int_0^{\infty}t\|\nabla u\|^2_{L^{\infty}}\,\dt\leq C\int_0^{\infty}t\bigl(\|\nabla u\|_{L^2}\|\nabla \dot{u}\|_{L^2}+\|\sqrt{\rho}\dot{u}\|^2_{L^2}\bigr)\,\dt\\
     &\leq C\|\nabla u\|_{L^2(\R^+;L^2)}\|t\nabla\dot{u}\|_{L^2(\R^+;L^2)}+C\|\sqrt{t\rho}\dot{u}\|^2_{L^2(\R^+;L^2)}\leq C.  
    \end{align*}
    
    This completes the proof of Lemma \ref{Lem.L^2_L^infty_2D}.
\end{proof}

\section{A priori estimates in the 3-D case}
In this section, we establish the estimates that are needed to prove Theorem \ref{Thm.existence}. Let $(\rho,u)$ be a smooth solution to the system \eqref{INS} on $[0,T_*)\times \R^3$ with $\rho_*< \rho\leq \|\rho_0\|_{L^{\infty}}$ for some constant $\rho_*>0$ (recalling footnote \ref{footnote_low-bound}) and $u_0\in\dot H^{1/2}(\R^3)$, where $T_*\in(0, +\infty]$ is the maximal time of existence.

\begin{lemma}[Estimates on the linear system]\label{Lem.linear}
	Assume that  $\|\nabla u\|_{L^4(0,T_*;L^2(\R^3))}\leq 1$. Then the following linear system for $(v, \nabla\pi)$
	\begin{equation}\label{Eq.linear_NS}
		\begin{cases}
			\rho(\pa_tv+u\cdot\nabla v)-\Delta v+\nabla \pi=0,\\
			\dive v=0,\\
			v|_{t=0}=v_0\in H^{1}(\R^3)
		\end{cases}
	\end{equation}
	admits a unique solution $v\in C([0, T_*); H^1(\R^3))$ 
on $[0,T_*)\times \R^3$ with the following estimates 
	\begin{align}
		\sup_{t\in[0,T_*)}\|\sqrt\rho v(t)\|_{L^2(\R^3)}+\|\nabla v\|_{L^2(0,T_*;L^2(\R^3))}&\leq C\|v_0\|_{L^2(\R^3)},\label{Eq.linear_est1}\\
		\sup_{t\in[0,T_*)}\|\nabla v(t)\|_{L^2(\R^3)}+\|\sqrt\rho\pa_tv\|_{L^2(0,T_*;L^2(\R^3))}&\leq C\|\nabla v_0\|_{L^2(\R^3)},\label{Eq.linear_est3}\\
       \sup_{t\in[0,T_*)}\|\sqrt t\nabla v\|_{ L^2(\R^3)}+\|\sqrt{t\rho}\pa_tv\|_{L^2(0,T_*;L^2(\R^3))}&\leq C\|v_0\|_{L^2(\R^3)},\label{Eq.linear_est2}\\
		\|t^{-\alpha}\nabla v\|_{L^2(0,T_*;L^2(\R^3))}\leq C_\alpha \|v_0\|_{L^2(\R^3)}^{1-2\alpha}&\|\nabla v_0\|_{L^2(\R^3)}^{2\alpha},\quad\forall\ \alpha\in(0,1/2),\label{Eq.linear_est4}
	\end{align}
	where $C>0$ is a constant depending only on $\|\rho_0\|_{L^{\infty}}$, and $C_\alpha>0$ depends only on $\|\rho_0\|_{L^{\infty}}$ and $\alpha\in(0,1/2)$.
\end{lemma}

\begin{proof}
	Taking $L^2$ inner product of the momentum equation of \eqref{Eq.linear_NS} with $v$ and using the transport equation of \eqref{INS}, one can show \eqref{Eq.linear_est1} along the same way as the proof of Proposition 2.1 in \cite{Zhang_Adv}. 
	
	Taking $L^2$ inner product of the momentum equation of \eqref{Eq.linear_NS} with $\pa_tv$ gives
	\begin{align}
		\frac12\frac{\mathrm d}{\dt}\|\nabla v(t)\|_{L^2}^2+\|\sqrt\rho\pa_tv(t)\|_{L^2}^2&=-\int_{\R^3}\rho(u\cdot\nabla v)\cdot\pa_tv(t,x)\,\dx\nonumber\\
		&\leq \|\rho_0\|_{L^\infty}^{1/2}\|u(t)\|_{L^6}\|\nabla v(t)\|_{L^3}\|\sqrt\rho\pa_tv(t)\|_{L^2}\nonumber\\
		&\leq C\|\nabla u(t)\|_{L^2}\|\nabla v(t)\|_{L^2}^{1/2}\|\nabla^2v(t)\|_{L^2}^{1/2}\|\sqrt\rho\pa_tv(t)\|_{L^2}.\label{Eq.inner-product_pa_tv}
	\end{align}
	Moreover, the Stokes estimate implies 
	\begin{align*}
		\|\nabla^2v(t)\|_{L^2}+\|\nabla \pi(t)\|_{L^2}&\leq C\left(\|\sqrt\rho\pa_tv(t)\|_{L^2}+\|\rho_0\|_{L^\infty}\|u\cdot\nabla v(t)\|_{L^2}\right)\\
		&\leq C\left(\|\sqrt\rho\pa_tv(t)\|_{L^2}+\|u(t)\|_{L^6}\|\nabla v(t)\|_{L^3}\right)\\
		&\leq C\left(\|\sqrt\rho\pa_tv(t)\|_{L^2}+\|\nabla u(t)\|_{L^2}\|\nabla v(t)\|_{L^2}^{1/2}\|\nabla^2v(t)\|_{L^2}^{1/2}\right).
	\end{align*}
	By Young's inequality, we have
	\[\|\nabla^2v(t)\|_{L^2}\leq C\left(\|\sqrt\rho\pa_tv(t)\|_{L^2}+\|\nabla u(t)\|_{L^2}^2\|\nabla v(t)\|_{L^2}\right),\quad\forall\ t\in[0, T_*).\]
	Thus, \eqref{Eq.inner-product_pa_tv} becomes
	\begin{align*}
		\frac12&\frac{\mathrm d}{\dt}\|\nabla v(t)\|_{L^2}^2+\|\sqrt\rho\pa_tv(t)\|_{L^2}^2\\
		&\leq C\|\nabla u(t)\|_{L^2}\|\nabla v(t)\|_{L^2}^{1/2}\|\sqrt\rho\pa_tv(t)\|_{L^2}\left(\|\sqrt\rho\pa_tv(t)\|_{L^2}^{1/2}+\|\nabla u(t)\|_{L^2}\|\nabla v(t)\|_{L^2}^{1/2}\right)\\
		&= C\|\nabla u(t)\|_{L^2}\|\nabla v(t)\|_{L^2}^{1/2}\|\sqrt\rho\pa_tv(t)\|_{L^2}^{3/2}+C\|\nabla u(t)\|_{L^2}^2\|\nabla v(t)\|_{L^2}\|\sqrt\rho\pa_tv(t)\|_{L^2}\\
		&\leq C\|\nabla u(t)\|_{L^2}^4\|\nabla v(t)\|_{L^2}^{2}+\frac12\|\sqrt\rho\pa_tv(t)\|_{L^2}^{2},
	\end{align*}
	where in the last inequality we have used Young's inequality. Hence, we have
	\begin{equation}\label{Eq.2.7}
		\frac{\mathrm d}{\dt}\|\nabla v(t)\|_{L^2}^2+\|\sqrt\rho\pa_tv(t)\|_{L^2}^2\leq C\|\nabla u(t)\|_{L^2}^4\|\nabla v(t)\|_{L^2}^{2},\quad\forall\ t\in[0, T_*).
	\end{equation}
	Then it follows from Gr\"onwall's inequality that
	\begin{equation*}
		\|\nabla v(t)\|_{L^2}+\int_0^t\|\sqrt\rho\pa_tv(s)\|_{L^2}^2\,\mathrm ds\leq\exp\left(C\int_0^t\|\nabla u\|_{L^2}^4\,\ds\right)\|\nabla v_0\|_{L^2}^2,\quad\forall\ t\in[0, T_*),
	\end{equation*}
	which gives \eqref{Eq.linear_est3}.
	
	Multiplying both sides of \eqref{Eq.2.7} by $t$, we obtain
	\begin{equation*}
		\frac{\mathrm d}{\dt}(t\|\nabla v(t)\|_{L^2}^2)+t\|\sqrt\rho\pa_tv(t)\|_{L^2}^2\leq\|\nabla v(t)\|_{L^2}^2+Ct\|\nabla u(t)\|_{L^2}^4\|\nabla v(t)\|_{L^2}^{2},\quad\forall\ t\in[0, T_*).
	\end{equation*}
	Then it follows from Gr\"onwall's inequality and \eqref{Eq.linear_est1} that
	\begin{equation*}
		t\|\nabla v(t)\|^2_{L^2}+\int_0^t\|\sqrt{s\rho}\pa_tv(s)\|_{L^2}^2\,\mathrm ds\leq\exp\left(C\|\nabla u\|_{L^4(0,t;L^2)}^4\right)\|\nabla v\|_{L^2(0,t;L^2)}^2\leq C\|v_0\|_{L^2}^2
	\end{equation*}
	for $t\in[0, T_*)$, which proves \eqref{Eq.linear_est2}.
	
	Finally, we prove \eqref{Eq.linear_est4}. It suffices to show that
	\begin{equation}\label{Eq.linear_est4claim}
		\int_0^ts^{-2\alpha}\|\nabla v(s)\|_{L^2}^2\,\mathrm ds\leq C_\alpha\|v_0\|_{L^2}^{2-4\alpha}\|\nabla v_0\|_{L^2}^{4\alpha},\quad\forall\ t\in[0, T_*).
	\end{equation}
    Let $t_0:=\|v_0\|_{L^2}^2\|\nabla v_0\|_{L^2}^{-2}$. If $0\leq t\leq t_0$, then we get by \eqref{Eq.linear_est3}  that 
    \begin{align*}
        \int_0^ts^{-2\alpha}\|\nabla v(s)\|_{L^2}^2\,\mathrm ds&\leq 
        C\int_0^ts^{-2\alpha}\|\nabla v_0\|_{L^2}^2\,\mathrm ds\leq C_\alpha t^{1-2\alpha}\|\nabla v_0\|_{L^2}^2\\
        &\leq C_\alpha t_0^{1-2\alpha}\|\nabla v_0\|_{L^2}^2=C_\alpha \|v_0\|_{L^2}^{2-4\alpha}\|\nabla v_0\|_{L^2}^{4\alpha}.
    \end{align*}
    If $t\geq t_0$, then by \eqref{Eq.linear_est1} and \eqref{Eq.linear_est3}, we have
    \begin{align*}
        \int_0^ts^{-2\alpha}\|\nabla v(s)\|_{L^2}^2\,\mathrm ds&\leq 
        C\int_0^{t_0}s^{-2\alpha}\|\nabla v_0\|_{L^2}^2\,\mathrm ds+t_0^{-2\alpha}\int_{t_0}^t\|\nabla v(s)\|_{L^2}^2\,\mathrm ds\\
        &\leq         C_\alpha t_0^{1-2\alpha}\|\nabla v_0\|_{L^2}^2+C t_0^{-2\alpha}\|v_0\|_{L^2}^2=(C+C_\alpha) \|v_0\|_{L^2}^{2-4\alpha}\|\nabla v_0\|_{L^2}^{4\alpha}.
    \end{align*}
    This proves \eqref{Eq.linear_est4claim}, and thus completes the proof of \eqref{Eq.linear_est4}.
\end{proof}

\begin{lemma}\label{Lem.low_order_est}
	There exists a constant $\varepsilon_0\in(0,1)$ depending only on $\|\rho_0\|_{L^\infty}$ such that if $\|u_0\|_{\dot H^{1/2}(\R^3)}<\varepsilon_0$, then we have 
	\begin{equation}\label{Eq.low_order_est}
		\|t^{1/4}\nabla u\|_{L^\infty(0,T_*;L^2(\R^3))}+\|t^{-1/4}\nabla u\|_{L^2(0,T_*;L^2(\R^3))}\leq C\|u_0\|_{\dot H^{1/2}(\R^3)}
	\end{equation}
	for some constant $C>0$ depending only on $\|\rho_0\|_{L^\infty}$.
\end{lemma}
\begin{proof}

    We denote
    \begin{equation}\label{Eq.T^*}
        T^*:=\sup\{t\in(0, T_*): \|\nabla u\|_{L^4(0, t; L^2)}\leq 1\}.
    \end{equation}
    Then we prove \eqref{Eq.low_order_est} on $(0, T^*)$.
	For each $j\in\Z$, we consider the following coupled system\footnote{In the proof of Lemma \ref{Lem.low_order_est} and Lemma \ref{Lem.L^3_est}, we use the same notations as in \cite{Zhang_Adv}. In particular, the definition of $\Delta_j$ can be found in Appendix A of \cite{Zhang_Adv}, and $(c_j)$ denotes a generic element of $\ell^2(\Z)$ such that $c_j\geq 0$ and $\sum_{j\in\Z}c_j^2=1$. Similarly for $(\tilde c_k)$ and $(\tilde c_k')$.} of $(u_j, \nabla P_j)$:
	 \begin{equation}\label{Eq.u_j_eq}
	 	\begin{cases}
	 		\rho(\pa_tu_j+u\cdot\nabla u_j)-\Delta u_j+\nabla P_j=0,\\
	 		\dive u_j=0,\\
	 		u_j|_{t=0}=\Delta_ju_0.
	 	\end{cases}
	 \end{equation}
	 Then it follows from the uniqueness of local smooth solution to \eqref{INS} that
    \begin{equation}\label{Eq.u=sum_u_j}
        u=\sum_{j\in\Z}u_j,\qquad \nabla P=\sum_{j\in\Z}\nabla P_j.
    \end{equation}
	 By Lemma \ref{Lem.linear}, there exists a constant $C>0$ such that for all $j\in\Z$, we have
	 \begin{align*}
	 	\|\nabla u_j\|_{L^\infty(0, T^*; L^2)}&\leq C\|\nabla \Delta_ju_0\|_{L^2}\leq C2^{j/2}c_j\|u_0\|_{\dot H^{1/2}},\\
	 	\|\sqrt{t}\nabla u_j\|_{L^\infty(0, T^*; L^2)}&\leq C\|\Delta_ju_0\|_{L^2}\leq C2^{-j/2}c_j\|u_0\|_{\dot H^{1/2}}.
	 \end{align*}
     Hence, for $t\in[0, T^*)$ we obtain
	 \begin{align*}
	 	\sqrt t\|\nabla u(t)\|_{L^2}^2&=\sqrt t\int_{\R^3}|\nabla u(t,x)|^2\,\dx
	 	=\sum_{j,k\in\Z}\int_{\R^3}\sqrt{t}\nabla u_j(t,x)\cdot\nabla u_k(t,x)\,\dx\\
	 	&\leq 2\sum_{k\in\Z}\sum_{j\leq k}\sqrt{t}\|\nabla u_k(t)\|_{L^2}\|\nabla u_j(t)\|_{L^2}\\
	 	&\leq C\sum_{k\in\Z}2^{-k/2}c_k\sum_{j\leq k}2^{j/2}c_j\|u_0\|_{\dot H^{1/2}}^2\leq C\sum_{k\in\Z}c_k\left(\sum_{j\leq k}2^{(j-k)/2}c_j\right)\|u_0\|_{\dot H^{1/2}}^2\\
	 	&\leq C\sum_{k\in\Z}c_k\tilde c_k\|u_0\|_{\dot H^{1/2}}^2\leq C\|u_0\|_{\dot H^{1/2}}^2.
	 \end{align*}
	 This shows that
	 \begin{equation}\label{Eq.low_order_est1}
	 	\|t^{1/4}\nabla u\|_{L^\infty(0,T^*; L^2)}\leq C\|u_0\|_{\dot H^{1/2}}.
	 \end{equation}
	 
	 Taking  $\varepsilon=1/8\in(0,1/4)$, Lemma \ref{Lem.linear} implies 
	 \begin{align*}
	 	\|t^{-(1/2-\varepsilon)}\nabla u_j\|_{L^2(0, T^*; L^2)}&\leq C\|\Delta_ju_0\|_{L^2}^{2\varepsilon}\|\nabla\Delta_ju_0\|_{L^2}^{1-2\varepsilon}\leq C2^{j(1/2-2\varepsilon)}c_j\|u_0\|_{\dot H^{1/2}},\\
	 	\|t^{-\varepsilon}\nabla u_j\|_{L^2(0, T^*; L^2)}&\leq C\|\Delta_ju_0\|_{L^2}^{1-2\varepsilon}\|\nabla\Delta_ju_0\|_{L^2}^{2\varepsilon}\leq C2^{j(2\varepsilon-1/2)}c_j\|u_0\|_{\dot H^{1/2}}
	 \end{align*}
	 for all $j\in\Z$. Hence, we obtain
	 \begin{align*}
	 	\int_0^{T^*}&t^{-1/2}\int_{\R^3}|\nabla u(t,x)|^2\,\dx\,\dt\leq \sum_{j,k\in\Z}\int_0^{T^*}t^{-1/2}\|\nabla u_j(t)\|_{L^2}\|\nabla u_k(t)\|_{L^2}\,\dt\\
	 	&\leq  2\sum_{k\in\Z}\sum_{j\leq k}\int_0^{T^*}\|t^{-(1/2-\varepsilon)}\nabla u_j(t)\|_{L^2}\|t^{-\varepsilon}\nabla u_k(t)\|_{L^2}\,\dt\\
	 	&\leq C\sum_{k\in\Z}2^{k(2\varepsilon-1/2)}c_k\sum_{j\leq k}2^{j(1/2-2\varepsilon)}c_j\|u_0\|_{\dot H^{1/2}}^2\\
	 	&\leq C\sum_{k\in\Z}c_k\left(\sum_{j\in\Z}2^{-|k-j|(1/2-2\varepsilon)}c_j\right)\|u_0\|_{\dot H^{1/2}}^2\\
	 	&\leq C\sum_{k\in\Z}c_k\tilde c_k'\|u_0\|_{\dot H^{1/2}}^2\leq C\|u_0\|_{\dot H^{1/2}}^2.
	 \end{align*}
	 This shows that
	 \begin{equation}\label{Eq.low_order_est2}
	 	\|t^{-1/4}\nabla u\|_{L^2(0, T^*; L^2)}\leq C\|u_0\|_{\dot H^{1/2}}.
	 \end{equation}

    It follows from \eqref{Eq.low_order_est1} and \eqref{Eq.low_order_est2} that
    \begin{equation}\label{Eq.nabla_u_bootstrap}
        \|\nabla u\|_{L^4(0, T^*; L^2)}\leq C^*\|u_0\|_{\dot H^{1/2}}
    \end{equation}
    for some constant $C^*>0$ depending only on $\|\rho_0\|_{L^\infty}$. Now we take $\varepsilon_0:=\min\{1/(2C^*), 1/2\}\in(0,1)$. Thus, if $\|u_0\|_{\dot H^{1/2}}<\varepsilon_0$, then by \eqref{Eq.T^*}, \eqref{Eq.nabla_u_bootstrap} and a continuity argument, we have $T^*=T_*$, and hence \eqref{Eq.low_order_est} follows from \eqref{Eq.low_order_est1} and \eqref{Eq.low_order_est2}.
\end{proof}

\begin{remark}\label{Rmk.global}
    By Lemma \ref{Lem.low_order_est}, if $\|u_0\|_{\dot H^{1/2}}<\varepsilon_0$, then $\|t^{1/4}\nabla u\|_{L^\infty(0, T_*; L^2)}\leq C\|u_0\|_{\dot H^{1/2}}$, hence we deduce from the classical theory of inhomogeneous Navier-Stokes equations that $T_*=+\infty$, i.e., the smooth solution $(\rho, u)$ to \eqref{INS} is globally defined. As a consequence, in this subsection, we have $T_*=+\infty$ and
    \begin{equation}\label{Eq.nabla_u_L4L2}
    	\|\nabla u\|_{L^4(\R^+; L^2)}<1.
    \end{equation}
\end{remark}

\begin{lemma}[Critical estimates]\label{Lem.L^3_est}
Let $\varepsilon_0\in(0,1)$ be given by Lemma \ref{Lem.low_order_est}. If $\|u_0\|_{\dot H^{1/2}(\R^3)}<\varepsilon_0$, then we have
    \begin{equation}\label{Eq.rhou_L3}
        \|\sqrt\rho u\|_{L^\infty(\R^+; L^3(\R^3))}\leq C\|u_0\|_{\dot H^{1/2}(\R^3)}
    \end{equation}
    for some constant $C>0$ depending only on $\|\rho_0\|_{L^{\infty}}$.
\end{lemma}
\begin{proof}
    Consider the system \eqref{Eq.u_j_eq} for $(u_j, \nabla P_j)$, then we have \eqref{Eq.u=sum_u_j}. Lemma \ref{Lem.linear} implies 
    \begin{align*}
        \|\sqrt\rho u_j\|_{L^\infty(\R^+;L^2)}&\leq C\|\Delta_j u_0\|_{L^2}\leq C2^{-j/2}c_j\|u_0\|_{\dot H^{1/2}},\\
        \|\nabla u_j\|_{L^\infty(\R^+;L^2)}&\leq C\|\nabla \Delta_j u_0\|_{L^2}\leq C2^{j/2}c_j\|u_0\|_{\dot H^{1/2}},\quad\forall\ j\in\Z.
    \end{align*}
    Thus, we have
    \begin{align*}
        &\|\sqrt\rho u_j\|_{L^\infty(\R^+;L^6)}\leq 
        C\|\nabla u_j\|_{L^\infty(\R^+;L^2)}\leq C2^{j/2}c_j\|u_0\|_{\dot H^{1/2}},\quad\forall\ j\in\Z.
    \end{align*}
    As a consequence, there holds for all $t\in\R^+$,
    \begin{align*}
        &\|\sqrt\rho u(t)\|_{L^3}^2=\|\rho |u|^2(t)\|_{L^{3/2}}
        \leq \sum_{i,j\in\Z}\|\rho |u_iu_j|(t)\|_{L^{3/2}}\\
        &\leq 2\sum_{i\in\Z}\sum_{j\leq i}\|\sqrt\rho u_i(t)\|_{L^2}\|\sqrt\rho u_j(t)\|_{L^6}
        \leq C\sum_{i\in\Z}\sum_{j\leq i}2^{-i/2}c_i2^{j/2}c_j\|u_0\|_{\dot H^{1/2}}^2\\
        &\leq C\sum_{i\in\Z}c_i\left(\sum_{j\leq i}2^{-(i-j)/2}c_j\right)
        \|u_0\|_{\dot H^{1/2}}^2\leq C\sum_{i\in\Z}c_i\tilde c_i\|u_0\|_{\dot H^{1/2}}^2
        \leq C\|u_0\|_{\dot H^{1/2}}^2.
    \end{align*}
    This completes the proof of \eqref{Eq.rhou_L3}.
\end{proof}

\begin{lemma}
	Let $\varepsilon_0\in(0,1)$ be given by Lemma \ref{Lem.low_order_est}. If $\|u_0\|_{\dot H^{1/2}(\R^3)}<\varepsilon_0$, then
	\begin{equation}\label{Eq.low_order_est_rhou}
		\|t^{1/4}\nabla u\|_{L^\infty(\R^+;L^2(\R^3))}+\|t^{1/4}(\sqrt\rho\dot u, \nabla^2u)\|_{L^2(\R^+;L^2(\R^3))}\leq C\|u_0\|_{\dot H^{1/2}(\R^3)}
	\end{equation}
	for some constant $C>0$ depending only on $\|\rho_0\|_{L^{\infty}}$.
\end{lemma}
\begin{proof}
	Taking $L^2$ inner product of the momentum equation of \eqref{INS} with $\dot u$, we obtain\footnote{Repeated indices represent  summation.}
	\begin{align*}
		0&=\|\sqrt\rho\dot u(t)\|_{L^2}^2+\langle\nabla u(t), \nabla\dot u(t)\rangle_{L^2}+\langle\nabla P(t), \dot u(t)\rangle_{L^2}\\
		&=\|\sqrt\rho\dot u(t)\|_{L^2}^2+\frac12\frac{\mathrm d}{\dt}\|\nabla u(t)\|_{L^2}^2+\int_{\R^3}\pa_iu^j\pa_iu^k\pa_ku^j\,\dx-\int_{\R^3}P\pa_ju^k\pa_ku^j\,\dx.
	\end{align*}
	Hence, we have
	\begin{align*}
		\frac12&\frac{\mathrm d}{\dt}\|\nabla u(t)\|_{L^2}^2+\|\sqrt\rho\dot u(t)\|_{L^2}^2\leq \left(\|\nabla u(t)\|_{L^6}+\|P(t)\|_{L^6}\right)\|\nabla u(t)\|_{L^3}\|\nabla u(t)\|_{L^2}\\
		&\leq C\left(\|\nabla^2u(t)\|_{L^2}+\|\nabla P(t)\|_{L^2}\right)\|\nabla^2u(t)\|_{L^2}^{1/2}\|\nabla u(t)\|_{L^2}^{3/2},\quad\forall\ t\in[0, +\infty).
	\end{align*}
	Recall the Stokes estimate 
	\begin{equation}\label{Eq.Stokes_est}
		\|\nabla^2u(t)\|_{L^2}+\|\nabla P(t)\|_{L^2}\leq C\|\rho\dot u(t)\|_{L^2}\leq C\|\sqrt\rho\dot u(t)\|_{L^2},\quad\forall\ t\in[0, +\infty).
	\end{equation}
	Thus, we have
	\begin{align*}
		\frac12\frac{\mathrm d}{\dt}\|\nabla u(t)\|_{L^2}^2+\|\sqrt\rho\dot u(t)\|_{L^2}^2&\leq C\|\nabla u(t)\|_{L^2}^{3/2}\|\sqrt\rho\dot u(t)\|_{L^2}^{3/2}\\
		&\leq C\|\nabla u(t)\|_{L^2}^6+\frac12\|\sqrt\rho\dot u(t)\|_{L^2}^2,\quad\forall\ t\in[0, +\infty).
	\end{align*}
	Or equivalently, $\frac{\mathrm d}{\dt}\|\nabla u(t)\|_{L^2}^2+\|\sqrt\rho\dot u(t)\|_{L^2}^2\leq C\|\nabla u(t)\|_{L^2}^6$, hence
	\begin{align*}
		\frac{\mathrm d}{\dt}\left(t^{1/2}\|\nabla u(t)\|_{L^2}^2\right)+t^{1/2}\|\sqrt\rho\dot u(t)\|_{L^2}^2\leq \frac12t^{-1/2}\|\nabla u(t)\|_{L^2}^2+Ct^{1/2}\|\nabla u(t)\|_{L^2}^6
	\end{align*}
	for all $t\in\R^+$. Now it follows from Gr\"onwall's lemma that
	\begin{equation}\label{3.22}
		\sqrt t\|\nabla u(t)\|_{L^2}^2+\int_0^t\sqrt s\|\sqrt\rho \dot u(s)\|_{L^2}^2\,\mathrm ds\leq \frac12\exp\left(C\|\nabla u\|_{L^4(0,t;L^2)}^4\right)\|s^{-1/4}\nabla u\|_{L^2(0,t;L^2)}^2.
	\end{equation}
	By \eqref{Eq.nabla_u_L4L2}, \eqref{Eq.low_order_est}, \eqref{3.22} and \eqref{Eq.Stokes_est}, we deduce \eqref{Eq.low_order_est_rhou}.
\end{proof}

\begin{lemma}
	Let $\varepsilon_0\in(0,1)$ be given by Lemma \ref{Lem.low_order_est}. If $\|u_0\|_{\dot H^{1/2}(\R^3)}<\varepsilon_0$, then
	\begin{equation}\label{Eq.high_order_est}
		\|t^{3/4}\sqrt\rho\dot u\|_{L^\infty(\R^+; L^2(\R^3))}+\|t^{3/4}\nabla\dot u\|_{L^2(\R^+;L^2(\R^3))}\leq C\|u_0\|_{\dot H^{1/2}(\R^3)}
	\end{equation}
	for some constant $C>0$ depending only on $\|\rho_0\|_{L^{\infty}}$.
\end{lemma}
\begin{proof}
	Applying the material derivative $D_t$ to the momentum equation of \eqref{INS}, we get 
	\begin{equation*}
			\rho(\pa_t\dot{u}+u\cdot\nabla \dot{u})-\Delta\dot{u}
			=-\pa_k(\pa_ku\cdot \nabla u)-\pa_k u\cdot \nabla \pa_ku
			-(\nabla P_t+u\cdot \nabla (\nabla P)).
	\end{equation*}
	Taking $L^2$ inner product with $\dot{u}$ of the above equation, we get by integration by parts and $\dive u=0$ that 
	\begin{equation}\label{Eq.Dt_energy}
			\frac12 \frac{\mathrm d}{\dt}\|\sqrt\rho\dot u\|_{L^2}^2+\|\nabla \dot{u}(t)\|^2_{L^2}
			=\int_{\R^3}(\pa_ku\cdot \nabla u)\cdot\pa_k\dot{u}\,\dx+\int_{\R^3}\pa_ku\cdot(\pa_ku\cdot\nabla\dot{u})\,\dx+J,
	\end{equation}
	where
	$$J(t)\eqdef -\int_{\R^3}(\nabla P_t+u\cdot \nabla (\nabla P))\cdot\dot{u}\,\dx,\quad\forall\ t\in[0, +\infty).$$
	Along the same lines as in the proof of Lemma 3.2 in \cite{Hao_Shao_Wei_Zhang}, we have\footnote{Here we stress that $\pa_iu^k\pa_ku^j\pa_ju^i=0 $ holds for $d=2$, but not for $d=3$.}
	\begin{equation}
		J(t)=\frac{\mathrm d}{\dt}\int_{\R^3} P\pa_iu^j\pa_ju^i\,\dx
		-3\int_{\R^3} P \pa_iu^j\pa_j\dot{u}^i\,\dx+2\int_{\R^3} P\pa_iu^k\pa_ku^j\pa_ju^i\,\dx.
	\end{equation}
	for all $t\in[0, +\infty)$.  Let
	\begin{equation}
		\Psi(t)\eqdef\|\sqrt\rho\dot u(t)\|_{L^2}^2-2\int_{\R^3} P\pa_iu^j\pa_ju^i\,\dx,\quad\forall\ t\in[0, +\infty).
	\end{equation}
	Then we have
	\begin{align*}
		&\frac12\Psi'(t)+\|\nabla \dot{u}(t)\|^2_{L^2}
=\int_{\R^3}(\pa_ku\cdot \nabla u)\cdot\pa_k\dot{u}\,\dx+\int_{\R^3}\pa_ku\cdot(\pa_ku\cdot\nabla\dot{u})\,\dx\\
&\qquad\qquad-3\int_{\R^3} P \pa_iu^j\pa_j\dot{u}^i\,\dx+2\int_{\R^3} P\pa_iu^k\pa_ku^j\pa_ju^i\,\dx\\
		\leq&\ C\|\nabla u(t)\|_{L^6}\|\nabla u(t)\|_{L^3}\|\nabla \dot u(t)\|_{L^2}+C\|P(t)\|_{L^6}\|\nabla u(t)\|_{L^3}\|\nabla \dot u(t)\|_{L^2}\\&+C\|P(t)\|_{L^6}\|\nabla u(t)\|_{L^3}^2\|\nabla u(t)\|_{L^6}\\
		\leq&\ C\|\nabla u(t)\|_{L^3}\|\nabla \dot u(t)\|_{L^2}\|(\nabla^2u(t), \nabla P(t))\|_{L^2}
+C\|\nabla u(t)\|_{L^3}^2\|(\nabla^2u(t), \nabla P(t))\|_{L^2}^2\\
		\overset{\eqref{Eq.Stokes_est}}{\leq}&C\|\nabla u(t)\|_{L^3}\|\nabla \dot u(t)\|_{L^2}\|\sqrt\rho\dot u(t)\|_{L^2}+C\|\nabla u(t)\|_{L^3}^2\|\sqrt\rho\dot u(t)\|_{L^2}^2\\
			\leq&\ \frac12\|\nabla \dot{u}(t)\|^2_{L^2}+C\|\nabla u(t)\|_{L^3}^2\|\sqrt\rho\dot u(t)\|_{L^2}^2,
		\end{align*}
		hence,
  \begin{align}
      \Psi'(t)&+\|\nabla \dot{u}(t)\|^2_{L^2}\leq C\|\nabla u(t)\|_{L^3}^2\|\sqrt\rho\dot u(t)\|_{L^2}^2,\nonumber\\
      (t^{3/2}\Psi)'(t)&+t^{3/2}\|\nabla \dot{u}(t)\|^2_{L^2}\leq \frac32t^{1/2}\Psi(t)+Ct^{3/2}\|\nabla u(t)\|_{L^3}^2\|\sqrt\rho\dot u(t)\|_{L^2}^2\label{Eq.Psi'(t)}
  \end{align}
  for all $t\in[0, +\infty)$.

        On the other hand, we have
		\begin{align*}
			&\left|\int_{\R^3} P\pa_iu^j\pa_ju^i\,\dx\right|\leq \|P(t)\|_{L^6}\|\nabla u(t)\|_{L^3}\|\nabla u(t)\|_{L^2}\\
			&\leq C\|\nabla P(t)\|_{L^2}\|\nabla u(t)\|_{L^2}^{3/2}\|\nabla^2u(t)\|_{L^2}^{1/2}
			\overset{\eqref{Eq.Stokes_est}}{\leq}C\|\nabla u(t)\|_{L^2}^{3/2}\|\sqrt\rho\dot u(t)\|_{L^2}^{3/2}\\
			&\leq \frac14\|\sqrt\rho\dot u(t)\|_{L^2}^2+C\|\nabla u(t)\|_{L^2}^6,
		\end{align*}
		thus,
		\begin{equation}\label{Eq.Psi_est}
			\frac12\|\sqrt\rho\dot u(t)\|_{L^2}^2-C\|\nabla u(t)\|_{L^2}^6\leq \Psi(t)\leq 2\|\sqrt\rho\dot u(t)\|_{L^2}^2+C\|\nabla u(t)\|_{L^2}^6, \quad\forall\ t\in[0, +\infty).
		\end{equation}

    It follows from \eqref{Eq.Psi'(t)} and \eqref{Eq.Psi_est} that
    \begin{align*}
        &(t^{3/2}\Psi)'(t)+t^{3/2}\|\nabla \dot{u}(t)\|^2_{L^2}\\
        &\qquad\leq Ct^{1/2}\|\sqrt\rho\dot u(t)\|_{L^2}^2+Ct^{1/2}\|\nabla u(t)\|_{L^2}^6
        +Ct^{3/2}\|\nabla u(t)\|_{L^3}^2\|\sqrt\rho\dot u(t)\|_{L^2}^2,\quad\forall\ t\in[0, +\infty).
    \end{align*}
    Then we have
    \begin{align}\label{Eq.t^3/2Psi}
        &t^{3/2}\Psi(t)+\int_0^ts^{3/2}\|\nabla \dot{u}(s)\|^2_{L^2}\,\ds\\
    \nonumber    \leq & C\int_0^ts^{1/2}\|\sqrt\rho\dot u(s)\|_{L^2}^2\,\ds\left(1+\sup_{0<s<t}s\|\nabla u(s)\|_{L^3}^2\right)
    +C\int_0^ts^{1/2}\|\nabla u(s)\|_{L^2}^6\,\ds
    \end{align}
    for all $t\in[0, +\infty)$. Moreover, it follows from \eqref{Eq.low_order_est_rhou}, \eqref{Eq.low_order_est} and \eqref{Eq.nabla_u_L4L2} that 
    \begin{align*}
        &
        \int_0^ts^{1/2}\|\sqrt\rho\dot u(s)\|_{L^2}^2\,\ds+\int_0^ts^{1/2}\|\nabla u(s)\|_{L^2}^6\,\ds\\
        \leq & C\|t^{1/4}\sqrt\rho\dot u\|_{L^2(\R^+;L^2)}^2+C\|t^{1/4}\nabla u\|_{L^\infty(\R^+; L^2)}^2\|\nabla u\|_{L^4(\R^+;L^2)}^4\leq C\|u_0\|_{\dot H^{1/2}}^2;
    \end{align*}
    and we also have (for all $s>0$, also using \eqref{Eq.low_order_est} and \eqref{Eq.Stokes_est})
    \begin{align*}
        s\|\nabla u(s)\|_{L^3}^2&\leq Cs^{1/4}\|\nabla u(s)\|_{L^2}s^{3/4}\|\nabla^2u(s)\|_{L^2}\leq C\|u_0\|_{\dot H^{1/2}}s^{3/4}\|\sqrt\rho\dot u(s)\|_{L^2}.
    \end{align*}
    Thus, \eqref{Eq.t^3/2Psi} implies that (for all $t>0$, also using $\|u_0\|_{\dot H^{1/2}}<\varepsilon_0<1$)
    \begin{equation}\label{3.30}
        t^{3/2}\Psi(t)+\int_0^ts^{3/2}\|\nabla \dot{u}(s)\|^2_{L^2}\,\ds\leq C\|u_0\|_{\dot H^{1/2}}^2(1+A(t)),
    \end{equation}
    where $A(t)\eqdef\sup_{0<s<t}s^{3/4}\|\sqrt\rho\dot u(s)\|_{L^2}$. Thus using \eqref{Eq.Psi_est}, \eqref{3.30} and  \eqref{Eq.low_order_est}, we obtain
    \begin{align*}
        &t^{3/2}\|\sqrt\rho\dot u(t)\|_{L^2}^2+\int_0^ts^{3/2}\|\nabla \dot{u}(s)\|^2_{L^2}\,\ds\leq C\|u_0\|_{\dot H^{1/2}}^2(1+A(t))+Ct^{3/2}\|\nabla u(t)\|_{L^2}^6\\ \leq&C\|u_0\|_{\dot H^{1/2}}^2(1+A(t))+C\|t^{1/4}\nabla u\|_{L^\infty(\R^+; L^2)}^6\leq C\|u_0\|_{\dot H^{1/2}}^2(1+A(t))+C\|u_0\|_{\dot H^{1/2}}^6,\\
        &A(t)^2+\int_0^ts^{3/2}\|\nabla \dot{u}(s)\|^2_{L^2}\,\ds\leq C\|u_0\|_{\dot H^{1/2}}^2(1+A(t))+C\|u_0\|_{\dot H^{1/2}}^6,\\
       &A(t)^2+\int_0^ts^{3/2}\|\nabla \dot{u}(s)\|^2_{L^2}\,\ds\leq C\|u_0\|_{\dot H^{1/2}}^2+C\|u_0\|_{\dot H^{1/2}}^4+C\|u_0\|_{\dot H^{1/2}}^6\leq  C\|u_0\|_{\dot H^{1/2}}^2.
    \end{align*}
    
    This completes the proof.
\end{proof}

\begin{corollary}[Estimates for $u_t$]
    Let $\varepsilon_0\in(0,1)$ be given by Lemma \ref{Lem.low_order_est}. If $\|u_0\|_{\dot H^{1/2}}<\varepsilon_0$, then
	\begin{equation}\label{Eq.high_order_est_ut}
		\|t^{1/4}\sqrt\rho u_t\|_{L^2(\R^+; L^2)}+\|t^{3/4}\sqrt\rho u_t\|_{L^\infty(\R^+; L^2)}+\|t^{3/4}\nabla u_t\|_{L^2(\R^+;L^2)}\leq C\|u_0\|_{\dot H^{1/2}}
	\end{equation}
	for some constant $C>0$ depending only on $\|\rho_0\|_{L^{\infty}}$.
\end{corollary}
\begin{proof}
Using the inequality $\|f\|_{L^\infty(\R^3)}\leq C\|\nabla f\|_{L^2(\R^3)}^{1/2}\|\nabla f\|_{L^6(\R^3)}^{1/2}$, we have
    \begin{align*}
     &\|\sqrt\rho u\cdot\nabla u(t)\|_{L^2}\leq C\|u(t)\|_{L^6}\|\nabla u(t)\|_{L^3}\leq C\|\nabla u(t)\|_{L^2}^{\f32}\|\nabla^2 u(t)\|_{L^2}^{\f12};\\
     &\|\nabla(u\cdot\nabla u)(t)\|_{L^2}\leq \|\nabla u(t)\|_{L^3}\|\nabla u(t)\|_{L^6}+\|u(t)\|_{L^\infty}\|\nabla^2 u(t)\|_{L^2}
     \leq C\|\nabla u(t)\|_{L^2}^{\f12}\|\nabla^2u(t)\|_{L^2}^{\f32}.
    \end{align*}
    By \eqref{Eq.low_order_est}, \eqref{Eq.low_order_est_rhou}, \eqref{Eq.Stokes_est} and \eqref{Eq.high_order_est}, we get
    \begin{align*}
        &\int_{\R^+}t^{\f12}\|\sqrt\rho u\cdot\nabla u(t)\|_{L^2}^2\,\dt
        \leq C\int_{\R^+}t^{\f12}\|\nabla u(t)\|_{L^2}^3\|\nabla^2 u(t)\|_{L^2}\,\dt\\
        &\leq C\|t^{1/4}\nabla u\|_{L^\infty(\R^+; L^2)}^2\|t^{-1/4}\nabla u\|_{L^2(\R^+; L^2)}\|t^{1/4}\nabla^2 u\|_{L^2(\R^+; L^2)}\leq C\|u_0\|_{\dot H^{1/2}}^4;\\
        &t^{\f34}\|\sqrt\rho u\cdot\nabla u(t)\|_{L^2}\leq Ct^{\f34}\|\nabla u(t)\|_{L^2}^{\f32}\|\nabla^2u(t)\|_{L^2}^{\f12}
        \leq C\|t^{\f14}\nabla u\|_{L^2}^{\f32}\|t^{\f34}\sqrt\rho\dot u\|_{L^2}^{\f12}\leq C\|u_0\|_{\dot H^{\f12}}^2;\\
        &\int_{\R^+}t^{3/2}\|\nabla(u\cdot\nabla u)(t)\|_{L^2}^2\,\dt\leq C\int_{\R^+}t^{3/2}\|\nabla u(t)\|_{L^2}\|\nabla^2u(t)\|_{L^2}^3\,\dt\\
        &\leq \|t^{1/4}\nabla u\|_{L^\infty(\R^+; L^2)}\|t^{3/4}\sqrt\rho\dot u\|_{L^\infty(\R^+; L^2)}\int_{\R^+}t^{1/2}\|\sqrt\rho\dot u(t)\|_{L^2}^2\,\dt\leq C\|u_0\|_{\dot H^{1/2}}^4.
    \end{align*}
    Therefore, \eqref{Eq.high_order_est_ut} follows from \eqref{Eq.low_order_est_rhou}, \eqref{Eq.high_order_est} and $u_t=\dot u-u\cdot\nabla u$.
\end{proof}

\begin{corollary}[Estimates for $\sqrt{\rho_0}u_t$]
    Let $\varepsilon_0\in(0,1)$ be given by Lemma \ref{Lem.low_order_est}. If\\ $\|u_0\|_{\dot H^{1/2}(\R^3)}<\varepsilon_0$, then
	\begin{equation}\label{Eq.rho_0ut_est}
		\|t^{1/4}\sqrt{\rho_0} u_t\|_{L^2(\R^+; L^2(\R^3))}\leq C\|u_0\|_{\dot H^{1/2}}
	\end{equation}
	for some constant $C>0$ depending only on $\|\rho_0\|_{L^{\infty}}$.
\end{corollary}
\begin{proof}
By the density equation of \eqref{INS} and \eqref{Eq.rhou_L3}, we get
\begin{align}\label{Eq.rho_W-13}
\|\rho(t)-\rho_0\|_{\dot{W}^{-1,3}}\leq \int_0^t \|\rho u(s)\|_{L^3}\,\ds\leq t\|\rho u\|_{L^{\infty}(0,t;L^3)}\leq Ct\|u_0\|_{\dot H^{1/2}}. 
\end{align}
Note that
\begin{align*}
\int_0^{\infty}\int_{\R^3} t^{\f12}\rho_0 |u_t|^2\,\dx\,\dt=\int_0^{\infty}\int_{\R^3} t^{\f12}\rho(t) |u_t|^2\,\dx\,\dt-\int_0^{\infty}\int_{\R^3} t^{\f12}(\rho(t)-\rho_0) |u_t|^2\,\dx\,\dt.  
\end{align*}
By duality and \eqref{Eq.rho_W-13}, we can obtain
\begin{align*}
\Bigl|\int_{\R^3} (\rho(t)-\rho_0) |u_t|^2\,\dx\Bigr|&\leq \|\rho(t)-\rho_0\|_{\dot{W}^{-1,3}}\||u_t|^2\|_{\dot{W}^{1,\f32}}\leq 2\|\rho(t)-\rho_0\|_{\dot{W}^{-1,3}}\|\nabla u_t\|_{L^2}\|u_t\|_{L^6}\\
&\leq Ct\|u_0\|_{\dot{H}^{1/2}}\|\nabla u_t\|^2_{L^2}.
\end{align*}
Therefore, \eqref{Eq.rho_0ut_est} follows from \eqref{Eq.high_order_est_ut}.
\end{proof}

\begin{lemma}\label{Lem.nabla_u_L^2L^infty}
    Let $\varepsilon_0\in(0,1)$ be given by Lemma \ref{Lem.low_order_est}. If $\|u_0\|_{\dot H^{1/2}(\R^3)}<\varepsilon_0$, then
	\begin{equation}\label{Eq.nabla_u_L^2L^infty}
		\|t^{1/2}\nabla u\|_{L^2(\R^+; L^\infty(\R^3))}\leq C\|u_0\|_{\dot H^{1/2}(\R^3)}
	\end{equation}
	for some constant $C>0$ depending only on $\|\rho_0\|_{L^{\infty}}$.
\end{lemma}
\begin{proof}
    Applying the Stokes estimate to $-\Delta u+\nabla P=-\rho \dot u$ gives 
	\begin{equation}\label{Eq.Stokes_est1}
		\|\nabla^2u(t)\|_{L^6}\leq C\|\sqrt\rho\dot u(t)\|_{L^6}\leq C\|\dot u(t)\|_{L^6}\leq C\|\nabla \dot u(t)\|_{L^2}\quad\forall\ t\in\R^+.
	\end{equation}
    By the inequality $\|f\|_{L^\infty(\R^3)}\leq C\|\nabla f\|_{L^2(\R^3)}^{1/2}\|\nabla f\|_{L^6(\R^3)}^{1/2}$ and \eqref{Eq.Stokes_est1}, for each $t\in\R^+$ we have
    \begin{align*}
        \int_0^ts\|\nabla u(s)\|_{L^\infty}^2\,\ds&\leq C\int_0^ts\|\nabla^2u(s)\|_{L^6}\|\nabla^2u(s)\|_{L^2}\,\ds\leq C\int_0^ts\|\nabla\dot u(s)\|_{L^2}\|\nabla^2u(s)\|_{L^2}\,\ds\\
        &\leq \|s^{1/4}\nabla^2u\|_{L^2(\R^+;L^2)}\|s^{3/4}\nabla\dot u\|_{L^2(\R^+;L^2)}\leq C\|u_0\|^2_{\dot H^{1/2}},
    \end{align*}
    where in the last inequality we have used \eqref{Eq.low_order_est_rhou} and \eqref{Eq.high_order_est}. 
\end{proof}

\section{Proof of global existence}

This section is devoted to proving the existence of a global solution to \eqref{INS} under our assumptions (both in dimension $d=2$ and $d=3$).
\begin{proof}[Proof of the existence part of Theorem \ref{Thm.existence_2D}]
Let $d=2$, $0\leq\rho_0\in L^\infty$ and $u_0\in L^2(\R^2)$. We consider (where $\varepsilon\in(0,1)$)
$$ u_0^{\epsilon}\in C^{\infty}(\R^2) \with  \dive u_0^{\epsilon}=0 \andf \rho_0^{\epsilon}\in C^{\infty}(\R^2) \with \epsilon\leq \rho_0^{\epsilon}\leq \|\rho_0\|_{L^\infty(\R^2)} $$
such that
\begin{equation*}
u_0^{\epsilon}\to u_0 \text{~in~} L^2(\R^2), \quad \rho_0^{\epsilon}\rightharpoonup \rho_0 \text{~in~} L^{\infty} \text{~weak-*},
\andf~ \rho_0^{\epsilon}\to \rho_0 \text{~in~} L^{p}_{\text{loc}}(\R^2) \text{~if~} p<\infty,
\end{equation*}
as $\varepsilon\to 0+$. In light of the classical strong solution theory for the system \eqref{INS}, there exists a unique global smooth solution $(\rho^{\epsilon}, u^{\epsilon}, P^{\epsilon})$
corresponding to data $(\rho_0^{\epsilon},u_0^{\epsilon})$. Thus, the triple $(\rho^{\epsilon}, u^{\epsilon}, P^{\epsilon})$  satisfies all the a priori estimates of this subsections uniformly with respect to $\epsilon\in(0,1)$. Hence, $(\rho^{\epsilon}, u^{\epsilon})$ converges weakly (or weakly-*) to a limit $(\rho, u)$ as $\varepsilon\to0+$, up to subsequence. To show that the limit solves \eqref{INS} weakly, in view of standard compactness arguments, it suffices to prove that
\begin{equation}\label{Eq.convergence}
    \lim_{\varepsilon\to0+}\int_0^t\langle \rho^\varepsilon u^\varepsilon\otimes u^\varepsilon(s)-\rho u\otimes u(s), \nabla\varphi\rangle\,\ds=0\quad\forall\ t>0
\end{equation}
for any divergence-free function $\varphi\in C_c^\infty([0, +\infty)\times\R^2)$. We fix $t>0$. Let $\eta>0$. Since
\begin{align*}
    |\langle \rho^\varepsilon u^\varepsilon\otimes u^\varepsilon(s), \nabla\varphi\rangle|&\leq \|\sqrt{\rho^\varepsilon}u^\varepsilon\|_{L^\infty(\R^+; L^2(\R^2))}^2\|\nabla \varphi\|_{L^\infty(\R^+\times\R^2)}\leq C,\\
    |\langle \rho u\otimes u(s), \nabla\varphi\rangle|&\leq \|\sqrt{\rho}u\|_{L^\infty(\R^+; L^2(\R^2))}^2\|\nabla \varphi\|_{L^\infty(\R^+\times\R^2)}\leq C,
\end{align*}
for all $\varepsilon\in(0,1)$, $s>0$, where $C>0$ is a constant depending only on $\|\sqrt{\rho_0}u_0\|_{L^2}$; taking $\delta_0:=\min\{t/2, \eta/(4C)\}>0$, we have
\[\int_0^{\delta_0}\left|\langle \rho^\varepsilon u^\varepsilon\otimes u^\varepsilon(s)-\rho u\otimes u(s), \nabla\varphi\rangle\right|\,\ds<\frac\eta2,\quad\forall\ \varepsilon\in(0,1).\]
To prove \eqref{Eq.convergence}, it suffices to show that
\begin{equation}\label{Eq.convergence_delta}
    \lim_{\varepsilon\to0+}\int_{\delta_0}^t\langle \rho^\varepsilon u^\varepsilon\otimes u^\varepsilon(s)-\rho u\otimes u(s), \nabla\varphi\rangle\,\ds=0.
\end{equation}
Indeed, assuming that $\operatorname{supp}_x\varphi(s)\subset B(0, R)$ for all $s\in[\delta_0, t]$, on one hand, by the weak convergence of $\rho^\varepsilon $ to $\rho $ in $L^2([0,T]\times B_R)$ (up to subsequence) and Lemma \ref{u_L^infty} (and the compact support property of $\varphi$),  
we have
\begin{equation}\label{Eq.convergence_delta1}
    \lim_{\varepsilon\to0+}\int_{\delta_0}^t\langle (\rho^\varepsilon-\rho) u\otimes u(s), \nabla\varphi\rangle\,\ds=0.
\end{equation}
On the other hand,  by Lemma \ref{Lem.high_order_est_2D}, Lemma \ref{u_L^infty} (and Lemma 2.2, Lemma A.2 of \cite{Hao_Shao_Wei_Zhang}, by adjusting $R>0$ to be larger is necessary) we know that  $u^\varepsilon$ is uniformly bounded in $H^1([\delta_0,t]\times B_R)$ (for fixed $\delta_0\in(0,t)$), hence the compact embedding $H^1([\delta_0,t]\times B_R)\ssubset L^2([\delta_0,t]\times B_R)$ implies $u^\varepsilon\to u$ in $L^2([\delta_0,t]\times B_R)$ up to subsequence,
 thus $u^\varepsilon\otimes u^\varepsilon\to u\otimes u$ in $L^1([\delta_0,t]\times B_R)$  
 as $\varepsilon\to0+$, which along with the uniform boundedness of $\rho^\varepsilon $ in $L^{\infty}$ implies that
\begin{equation}\label{Eq.convergence_delta2}
    \lim_{\varepsilon\to0+}\int_{\delta_0}^t\left|\langle \rho^\varepsilon (u^\varepsilon\otimes u^\varepsilon-u\otimes u)(s), \nabla\varphi\rangle\right|\,\ds=0.
\end{equation}
Now \eqref{Eq.convergence_delta} follows from \eqref{Eq.convergence_delta1} and \eqref{Eq.convergence_delta2}. 
\end{proof}

\begin{proof}[Proof of the existence part of Theorem \ref{Thm.existence}]
Let $d=3$. Let $\varepsilon_0\in(0,1)$ be given by Lemma \ref{Lem.low_order_est}. Assume that $u_0\in\dot H^{1/2}(\R^3)$ with 
$\|u_0\|_{\dot H^{1/2}}<\varepsilon_0$ and $0\leq \rho_0\leq \|\rho_0\|_{L^{\infty}}$ for some constant $\|\rho_0\|_{L^{\infty}}>0$ with $\rho_0\not\equiv0$. We are going to show that there exists a global weak solution $(\rho, u, \nabla P)$ to \eqref{INS} with initial data $(\rho_0, u_0)$. The  idea is to take advantage of classical result to construct smooth solutions corresponding to smoothed-out approximate data with no vacuum, then to pass to the limit. More precisely, for $\delta\in(0,1/2)$, we consider
\[u_0^\delta \in C^\infty(\R^3)\with \dive u_0^\delta=0 \andf  \rho_0^\delta\in C^\infty(\R^3)\with \delta\leq \rho_0^\delta\leq \|\rho_0\|_{L^{\infty}}\]
such that
\[u_0^\delta\to u_0 \text{ in } \dot H^{1/2},\quad \|u_0^\delta\|_{\dot H^{1/2}}<\varepsilon_0, \quad \rho_0^{\delta}\rightharpoonup \rho_0 \text{~in~} L^{\infty} \text{~weak-*},\andf~ \rho_0^{\delta}\to \rho_0 \text{~in~} L^{p}_{\text{loc}} \text{~if~} p<\infty,\]
as $\delta\to 0+$. By the classical theory of (INS), for each $\delta\in(0,1/2)$, \eqref{INS} has a unique local smooth solution $(\rho^\delta, u^\delta)$ on $[0, T_*^\delta)$ with initial data $(\rho_0^\delta, u_0^\delta)$. We also have \eqref{Eq.low_order_est} for $(\rho^\delta, u^\delta)$. Thus, $\|u_0^\delta\|_{\dot H^{1/2}}<\varepsilon_0$ and Remark \ref{Rmk.global} implies that $T_*^\delta=+\infty$, i.e., $(\rho^\delta, u^\delta)$ is a global smooth solution. Moreover, we have \eqref{Eq.low_order_est_rhou}, \eqref{Eq.high_order_est} for $(\rho^\delta, u^\delta)$. Here we recall that the constants in \eqref{Eq.low_order_est}, \eqref{Eq.low_order_est_rhou}, \eqref{Eq.high_order_est}, \eqref{Eq.high_order_est_ut} and \eqref{Eq.nabla_u_L^2L^infty} depend only on $\|\rho_0\|_{L^{\infty}}$, so these inequalities are uniform in $\delta\in(0,1/2)$. As a consequence, a standard compactness argument (see \cite{Paicu_Zhang_Zhang_CPDE}) combining with an argument similar to the 2-D case (here we use $L^6$ for $d=3$ instead of $L^\infty$ for $d=2$ above) implies the convergence of $(\rho^\delta, u^\delta)$ to a weak solution $(\rho, u)$ of \eqref{INS} with initial data $(\rho_0, u_0)$. This completes the proof of the existence part of Theorem \ref{Thm.existence}.
\end{proof}

\section{Proof of  the uniqueness}

The goal of this section is to prove Theorem \ref{Thm.uniqueness}, and to show that it implies the uniqueness part of Theorem \ref{Thm.existence_2D} and Theorem \ref{Thm.existence}.
\subsection{Preliminary lemmas}
\begin{lemma}\label{Lem.transport}
    Fix $t>0$. Assume that $\bar u=\bar u(s,x):(0, +\infty)\times\R^d\to\R^d$  satisfies $s^{1/2}\nabla{\bar u}\in L^2(\R^+; L^{\infty}(\R^d))$. Let $p\in(1,+\infty)$ and $\varphi(x)\in \dot W^{1,p}(\R^d)$. Then there is a unique solution $\phi=\phi(s,x):(0, t]\times\R^d\to\R^d$ to
    \begin{equation}\label{Eq.transport_eq}
        \pa_s\phi+\bar u\cdot\nabla \phi=0,\quad (s,x)\in (0, t]\times\R^d,\qquad \phi(t,x)=\varphi(x),\quad x\in\R^d
    \end{equation}
    such that
    \begin{equation}\label{Eq.transport_est}
        \|\nabla \phi(s)\|_{L^p(\R^d)}\leq C\|\nabla \varphi\|_{L^p(\R^d)}\mathrm{e}^{C|\ln(t/s)|^{1/2}}\quad \forall\ s\in(0, t],
    \end{equation}
    where $C>0$ is a constant depending only on $\|s^{1/2}\nabla {\bar u}\|_{L^2(\R^+; L^{\infty}(\R^d))}$.
\end{lemma}
\begin{proof}
    The classical theory on transport equations (\cite{BCD}, Theorem 3.2) implies the existence and uniqueness of the solution $\phi$, with the estimate
    \begin{equation*}
        \|\nabla \phi(s)\|_{L^p(\R^d)}\leq C\|\nabla \varphi\|_{L^p(\R^d)}\exp\left(\int_s^t\|\nabla \bar{u}(\tau)\|_{L^\infty(\R^d)}\,\mathrm d\tau\right),\quad \forall\ s\in(0, t]
    \end{equation*}
    for some constant $C>0$. By $\tau^{1/2}\nabla\bar u\in L^2(\R^+;L^{\infty}(\R^d))$ we have
\begin{align*}
\int_s^t \|\nabla \bar{u}(\tau)\|_{L^{\infty}}\,\mathrm{d}\tau&=\int_s^t \tau^{-\f12}\cdot\tau^{\f12}\|\nabla \bar{u}(\tau)\|_{L^{\infty}}\,\mathrm{d}\tau
\leq |\ln(t/s)|^{\f12}\|\tau^{\f12}\nabla \bar{u}(\tau)\|_{L^2(0,t;L^{\infty})}\nonumber\\
&\leq C|\ln(t/s)|^{\f12}.
\end{align*}

This completes the proof.
\end{proof}

\begin{lemma}\label{Lem.Hardy}
    Let $T>0$ and $f:(0, T]\to [0, +\infty]$ be such that $f\in L^4(0, T)$. Let $A>0$ and
    \begin{equation}\label{Eq.F(t)_def}
        F(t):=\frac1t\int_0^tf(s)\mathrm e^{A|\ln(t/s)|^{1/2}}\,\ds\quad \forall\ t\in(0, T].
    \end{equation}
    Then we have $F\in L^4(0, T)$ with
    \begin{equation}\label{Eq.F_L^4}
        \|F\|_{L^4(0, T)}\leq C_A\|f\|_{L^4(0, T)},
    \end{equation}
    where $C_A>0$ is a constant depending only on $A$.
\end{lemma}
\begin{proof}
    By \eqref{Eq.F(t)_def}, we have
    \[F(t)=\int_0^1f(\tau t)\mathrm e^{A|\ln\tau|^{1/2}}\,\mathrm d\tau,\quad\forall\ t\in(0, T].\]
    Hence
    \begin{align*}
        \|F\|_{L^4(0, T)}\leq \int_0^1\|f(\tau \cdot)\|_{L^4(0, T)}\mathrm e^{A|\ln\tau|^{1/2}}\,\mathrm d\tau=\int_0^1\|f\|_{L^4(0, T)}\tau^{-1/4}\mathrm e^{A|\ln\tau|^{1/2}}\,\mathrm d\tau.
    \end{align*}
    Note that there exists $\tau_0=\tau_0(A)\in(0,1)$ such that $A|\ln(1/\tau)|^{1/2}\leq \frac12|\ln(1/\tau)|$ for $\tau\in (0,\tau_0)$, then
    \[\int_0^{\tau_0}\tau^{-1/4}\mathrm e^{A|\ln\tau|^{1/2}}\,\mathrm d\tau\leq \int_0^{\tau_0}\tau^{-1/4}\mathrm e^{\ln (\tau^{-1/2})}\,\mathrm d\tau=\int_0^{\tau_0}\tau^{-3/4}\,\mathrm d\tau=4\tau_0^{1/4},\]
    thus
    \[\int_0^1\tau^{-1/4}\mathrm e^{A|\ln\tau|^{1/2}}\,\mathrm d\tau\leq 4\tau_0^{1/4}+\int_{\tau_0}^1\tau^{-1/4}\mathrm e^{A|\ln\tau|^{1/2}}\,\mathrm d\tau=C_A.\]
    
    This completes the proof.
\end{proof}

\subsection{Uniqueness in the 2-D case of Theorem \ref{Thm.uniqueness}}

\begin{proof}[Proof of Theorem \ref{Thm.uniqueness} for $d=2$]
    Assume that $0<\rho_0\in L^\infty(\R^2)$ has a positive lower bound and $u_0\in L^2(\R^2)$. Let $(\rho, u, \nabla P)$ and $(\bar \rho, \bar u, \nabla\bar P)$ be two weak solutions to \eqref{INS} with the same initial data $(\rho_0, u_0)$ and satisfies all the hypotheses listed in Theorem \ref{Thm.uniqueness}. We denote $\delta\!\rho=\rho-\bar\rho$ and $\delta\!u=u-\bar u$. Then $(\delta\!\rho, \delta\!u)$ satisfies 
    \begin{equation}\label{Eq.delta u}
     \left\{
     \begin{array}{l}
     \partial_t\delta\!\rho+\bar{u}\cdot\nabla \delta\!\rho+\delta\!u\cdot\nabla\rho=0,\qquad (t,x)\in \mathbb{R}^{+}\times\mathbb{R}^{d},\\
     \rho(\partial_t\delta\!u+u\cdot\nabla \delta\!u)-\Delta \delta\!u+\nabla \delta\!P=-\delta\!\rho\dot{\bar{u}}-\rho\delta\!u\cdot\nabla\bar{u},\\
     \dive \delta\!u = 0,\\
     (\delta\!\rho, \delta\!u)|_{t=0} =(0, 0).
     \end{array}
     \right.
     \end{equation}
Here $\dot{\bar{u}}:=\pa_t\bar{u}+\bar{u}\cdot\nabla\bar{u}$. Let $T_1\in(0, T)$ be small enough. For all $t\in (0, T_1]$, we set
    \begin{equation}\label{Eq.D_t}
        D(t):=\sup_{s\in [0,t]}\|\sqrt{\rho}\delta\!u(s)\|_{L^2(\R^2)}+\|\nabla \delta\!u\|_{L^2(0,t;L^2(\R^2))}.
    \end{equation}
    Testing \eqref{Eq.delta u} against $\delta\!u$ gives 
    \begin{equation}
        \frac12 \frac{\mathrm d}{\dt}\|\sqrt\rho\delta\!u\|_{L^2}^2+\|\nabla \delta\!u(t)\|^2_{L^2}
			=-\mathrm I(t)-\mathrm{II}(t) \quad\forall\ t\in(0, T_1],
    \end{equation}
    where 
    \[\mathrm I(t):=\langle\delta\!\rho(t), \dot{\bar u}\cdot\delta\!u(t)\rangle,\quad \mathrm{II}(t):=\int_{\R^2}(\rho\delta\!u\cdot\nabla\bar u)\cdot \delta u(t,x)\,\dx,\quad \forall\ t\in(0, T_1].\]
    Since (thanks to the positive lower bound of $\rho$)
    \begin{align*}
        |\mathrm{II}(t)|&\leq \|\sqrt\rho\delta\!u(t)\|_{L^4}^2\|\nabla \bar u(t)\|_{L^2}\leq C\|\sqrt\rho\delta\!u(t)\|_{L^2}\|\nabla \delta\!u(t)\|_{L^2}\|\nabla \bar u(t)\|_{L^2}\\
        &\leq \f12\|\nabla \delta\!u(t)\|_{L^2}^2
        +C\|\sqrt\rho\delta\!u(t)\|_{L^2}^2\|\nabla \bar u(t)\|_{L^2}^2,
    \end{align*}
    we have
    \[\frac{\mathrm d}{\dt}\|\sqrt\rho\delta\!u\|_{L^2}^2+\|\nabla \delta\!u(t)\|^2_{L^2}\leq 2|\mathrm I(t)|+C\|\sqrt\rho\delta\!u(t)\|_{L^2}^2\|\nabla \bar u(t)\|_{L^2}^2,\quad\forall\ t\in(0, T_1].\]
    It follows from Gr\"onwall's lemma  and $\nabla\bar u\in L^2(0, T; L^2(\R^2))$ that
    \begin{equation*}
        \|\sqrt\rho\delta\!u(t)\|_{L^2}^2+\|\nabla \delta\!u\|^2_{L^2(0, t; L^2)}\leq C_*\int_0^t|\mathrm I(s)|\,\ds,\quad\forall\ t\in(0, T_1],
    \end{equation*}
    where $C_*>0$ depends only on $\|\nabla\bar u\|_{L^2(0, T; L^2)}$ and $\|\rho_0\|_{L^{\infty}}$. Thus,
    \begin{equation}\label{Eq.D_leq_I}
        D(t)^2\leq C_*\int_0^t|\mathrm I(s)|\,\ds,\quad\forall\ t\in(0, T_1].
    \end{equation}

    For any fixed $t\in(0, T_1]$, by Lemma \ref{Lem.transport}, there exists a unique $\phi(s,x;t)$ solving
    \[\pa_s\phi(s,x;t)+\bar u(s,x)\cdot\nabla \phi(s, x; t)=0,\ \ (s,x)\in(0, t]\times\R^2,\quad \phi(t,x;t)=(\dot{\bar u}\cdot\delta u)(t,x)\]
    with the estimate
    \begin{equation}\label{Eq.nabla_phi_2D}
        \|\nabla\phi(s,\cdot;t)\|_{L^{4/3}}\leq C\|\nabla(\dot{\bar u}\cdot\delta\!u)(t)\|_{L^{4/3}}\,\mathrm e^{C|\ln(t/s)|^{1/2}},\quad\forall\ s\in(0, t].
    \end{equation}
    Using the first equation of \eqref{Eq.delta u} and integration by parts, we have
    \begin{align}
\frac{\mathrm{d}}{\ds} \langle\delta\!\rho(s), \phi(s,\cdot;t)\rangle
&=\langle -\dive\{ \delta\!\rho\bar{u}\},\phi\rangle-\langle \dive \{\rho \delta\!u\},\phi\rangle-\langle \delta\!\rho,\bar{u}\cdot\nabla\phi\rangle\nonumber\\
&=\langle \rho\delta\!u,\nabla\phi(s,\cdot;t)\rangle,\quad\forall\ s\in(0, t]\label{Eq.1_2D}
\end{align}
then integrating over $(0, t)$, by \eqref{Eq.nabla_phi_2D} we have
    \begin{align*}
        |\mathrm I(t)|&=|\langle\delta\!\rho(t), \phi(t,\cdot;t)\rangle|\leq \int_0^t|\langle\rho\delta\!u(s),\nabla\phi(s,\cdot;t)\rangle|\,\ds\leq \int_0^t\|\rho\delta\!u(s)\|_{L^4}\|\nabla\phi(s,\cdot;t)\|_{L^{4/3}}\,\ds\\
        &\leq \|\nabla(\dot{\bar u}\cdot\delta\!u)(t)\|_{L^{4/3}}\int_0^t \|\rho\delta\!u(s)\|_{L^4}\,\mathrm{e}^{C|\ln(t/s)|^{1/2}}\,\ds\\
        &=\|t\nabla(\dot{\bar u}\cdot\delta\!u)(t)\|_{L^{4/3}}\cdot\frac1t\int_0^t \|\rho\delta\!u(s)\|_{L^4}\,\mathrm{e}^{C|\ln(t/s)|^{1/2}}\,\ds=:\|t\nabla(\dot{\bar u}\cdot\delta\!u)(t)\|_{L^{4/3}}\cdot F(t)
    \end{align*}
    for all $t\in(0, T_1]$. Lemma \ref{Lem.Hardy} implies that 
    \begin{equation}
        \|F\|_{L^4(0, t)}\leq C\|\|\rho\delta\!u(s)\|_{L^4}\|_{L^4(0, t)}=C\|\rho\delta\!u\|_{L^4(0, t; L^4)},\quad\forall\ t\in(0, T_1].
    \end{equation}
    Hence by H\"older's inequality we obtain
    \begin{equation}\label{Eq.I_integral_est}
        \int_0^t|I(s)|\,\ds\leq \|s\nabla(\dot{\bar u}\cdot\delta\!u)\|_{L^{4/3}(0, t; L^{4/3})}\|\rho\delta\!u\|_{L^4(0, t; L^4)}, \quad\forall\ t\in(0, T_1].
    \end{equation}

    Moreover, by H\"older's inequality and Sobolev embedding we have
    \begin{align*}
        &\|s\nabla(\dot{\bar u}\cdot\delta\!u)\|_{L^{4/3}(0, t; L^{4/3})}=\|s(\nabla\dot{\bar u}\cdot\delta\!u+\dot{\bar u}\cdot\nabla\delta\!u)\|_{L^{4/3}(0, t; L^{4/3})}\\
        &\leq  \|s\nabla\dot{\bar u}\|_{L^{2}(0, t; L^{2})}\|\delta\!u\|_{L^{4}(0, t; L^{4})}+\|s\dot{\bar u}\|_{L^{4}(0, t; L^{4})}\|\nabla\delta\!u\|_{L^{2}(0, t; L^{2})},
    \end{align*}
    and (since $\rho(t)\in L^\infty(\R^2)$ is bounded away from zero)
    \begin{align*}
        \|\delta\!u\|_{L^4(0, t; L^4)}\leq C\left\|\|\sqrt\rho\delta\!u\|_{L^2}^{1/2}\|\nabla\delta\!u\|_{L^2}^{1/2}\right\|_{L^4(0, t)}\leq C\|\sqrt\rho\delta\!u\|_{L^\infty(0, t;L^2)}^{1/2}\|\nabla\delta\!u\|_{L^2(0, t; L^2)}^{1/2}\leq CD(t)
    \end{align*}
    for all $t\in(0, T_1]$, then we get (using the fact that $\rho(t)\in L^\infty(\R^2)$ is bounded away from zero) 
 \begin{align*}
        \|s\nabla(\dot{\bar u}\cdot\delta\!u)\|_{L^{4/3}(0, t; L^{4/3})}
        &\leq C D(t)\left(\|s\nabla\dot{\bar u}\|_{L^{2}(0, t; L^{2})}+\|s\dot{\bar u}\|_{L^{4}(0, t; L^{4})}\right)\\
        &\leq CD(t)\left(\|s\nabla\dot{\bar u}\|_{L^{2}(0, t; L^{2})}+\|s\dot{\bar u}\|_{L^{\infty}(0, t; L^{2})}^{1/2}\|s\nabla\dot{\bar u}\|_{L^{2}(0, t; L^{2})}^{1/2}\right)
    \end{align*}
    for all $t\in(0, T_1]$. Thus, \eqref{Eq.I_integral_est} implies that
    \begin{align}\label{Eq.I_integral_est2}
        \int_0^t|I(s)|\,\ds\leq C D(t)^2 \|s\nabla\dot{\bar u}\|_{L^{2}(0, t; L^{2})}^{1/2},\quad\forall\ t\in(0, T_1]
    \end{align}
    for some constant $C>0$ independent of $t$ and $T_1$. Combining \eqref{Eq.D_leq_I} and \eqref{Eq.I_integral_est2}, we obtain
    \begin{equation}\label{Eq.D_leq_D}
        D(t)^2\leq CD(t)^2\|s\nabla\dot{\bar u}\|_{L^{2}(0, t; L^{2})}^{1/2},\quad\forall\ t\in(0, T_1].
    \end{equation}
    As $t\nabla \dot{\bar u}\in L^2(0, T; L^2)$, we can take $T_1>0$ small enough such that $$C\|s\nabla\dot{\bar u}\|_{L^{2}(0, T_1; L^{2})}^{1/2}<1/2,$$
    then \eqref{Eq.D_leq_D} implies $0\leq D(t)^2\leq D(t)^2/2$ for all $t\in(0, T_1]$, hence $D(t)\equiv0$ for $t\in(0, T_1]$.  Therefore, we have $\delta\! u\equiv0$ on $[0, T_1]$. Then by the first equation of \eqref{Eq.delta u}, we get $\delta\!\rho\equiv0$ on $[0, T_1]$. This proves the uniqueness on $[0, T_1]$. The uniqueness on the whole time interval $[0,+\infty)$ can be obtained by a bootstrap argument.
\end{proof}

\subsection{Uniqueness in the 3-D case of Theorem \ref{Thm.uniqueness}}
\begin{proof}[Proof of Theorem \ref{Thm.uniqueness} for $d=3$]
Assume that $0\leq\rho_0\in L^\infty(\R^3)$ and $u_0\in \dot H^{1/2}(\R^3)$. Let $(\rho, u, \nabla P)$ and $(\bar \rho, \bar u, \nabla\bar P)$ be two weak solutions to \eqref{INS} with the same initial data $(\rho_0, u_0)$ and satisfies all the hypotheses listed in Theorem \ref{Thm.uniqueness}. We denote $\delta\!\rho:=\rho-\bar{\rho}$ and $\delta\!u:=u-\bar{u}$. Then ($\delta\!\rho,\delta\!u$) satisfies \eqref{Eq.delta u}.
Let $T>0$. For all $t\in(0,T]$, set
\begin{align}
&X(t):=\sup_{s\in (0,t]}s^{-\f34}\|\delta\!\rho(s)\|_{\dot{W}^{-1,3}(\R^3)},\label{X_t}\\
&Y(t):=\Bigl(\sup_{s\in [0,t]}\|\sqrt{\rho}\delta\!u(s)\|^2_{L^2(\R^3)}+\|\nabla \delta\!u\|^2_{L^2(0,t;L^2(\R^3))}\Bigr)^{1/2}.\label{Y_t}
\end{align}
For $\varphi\in C^1_c(\R^3)$, by Lemma \ref{Lem.transport}, there exists a unique $\phi(s,x)$ solving
\begin{align*}
\partial_s\phi+\bar{u}\cdot\nabla\phi=0,\quad (s,x)\in(0,t]\times\R^3, \qquad \phi(t,x)=\varphi(x),  
\end{align*}
with the estimate
\begin{align}\label{transport}
\|\nabla\phi(s)\|_{L^{\f32}}\leq C\|\nabla\varphi\|_{L^{\f32}}e^{C|\ln(t/s)|^{1/2}},\quad\forall\ s\in(0, t].
\end{align}
Note that
\begin{align}\label{Eq.ln_integrable}
\int_0^t  e^{C|\ln(t/s)|^{1/2}} \,\ds=t\int_0^1 e^{C|\ln(\tau)|^{1/2}} \,\mathrm{d}\tau\leq Ct\int_0^1\tau^{-\f12}\,\mathrm{d}\tau\leq Ct.  
\end{align}
Similarly with \eqref{Eq.1_2D}, we have
\begin{align*}
\frac{\mathrm{d}}{\ds} \langle\delta\!\rho(s), \phi(s)\rangle
&=\langle -\dive\{ \delta\!\rho\bar{u}\},\phi\rangle-\langle \dive \{\rho \delta\!u\},\phi\rangle-\langle \delta\!\rho,\bar{u}\cdot\nabla\phi\rangle
=\langle \rho\delta\!u,\nabla\phi\rangle.
\end{align*}
Integrating over $(0,t)$, and using \eqref{transport} and \eqref{Eq.ln_integrable} yield that
\begin{align}
|\langle\delta\!\rho(t), \phi(t)\rangle|&\leq
\int_0^t |\langle \rho \delta\!u,\nabla\phi\rangle|\,\ds \leq \int_0^t \|\rho\delta\!u(s)\|_{L^3}\|\nabla\phi\|_{L^{\f32}}\,\ds\nonumber\\
&\leq Ct^{\f34}\|\nabla\varphi\|_{L^{\f32}}\|\rho\delta\!u\|_{L^4(0,t;L^3)}\label{a}.
\end{align}
By H\"older's inequality and Sobolev embedding, we have
\begin{align*}
\int_0^t\|\rho\delta\!u\|^4_{L^3}\,\ds \leq \int_0^t\|\rho\delta\!u\|^2_{L^2}\|\rho\delta\!u\|^2_{L^6}\,\ds\leq C\|\sqrt{\rho}\delta\!u\|^2_{L^{\infty}(0,t;L^2)}\|\nabla \delta\!u\|^2_{L^2(0,t;L^2)}, 
\end{align*}
then it follows from \eqref{X_t}, \eqref{Y_t} and \eqref{a} that
\begin{align}\label{Eq.X<Y}
  X(t)\leq CY(t),\quad\forall\ t\in(0, T].  
\end{align}

Next, testing \eqref{Eq.delta u} against $\delta\!u$ gives
\begin{equation}\label{Eq.delta_u_energy} 
			\frac12 \frac{\mathrm d}{\dt}\|\sqrt\rho\delta\!u\|_{L^2}^2+\|\nabla \delta\!u(t)\|^2_{L^2}
			=-\int_{\R^3}\delta\!\rho\dot{\bar{u}}\cdot\delta\!u\,\dx-\int_{\R^3}(\rho\delta\!u\cdot\nabla\bar{u})\cdot\delta\!u\,\dx.
	\end{equation}
By H\"older's inequality and Sobolev embedding, we have
\begin{equation}\begin{aligned}\label{uniqueness.1}
 \left|\int_{\R^3}(\rho\delta\!u\cdot\nabla\bar{u})\cdot\delta\!u\,\dx\right|&\leq \|\rho\delta\!u\|_{L^3}\|\nabla\bar{u}\|_{L^2}\|\delta\!u\|_{L^6}\leq C\|\sqrt{\rho}\delta\!u\|^{\f12}_{L^2}\|\nabla\delta\!u\|^{\f32}_{L^2}\|\nabla\bar{u}\|_{L^2}\\
 &\leq \f14 \|\nabla\delta\!u\|^{2}_{L^2}+C \|\sqrt{\rho}\delta\!u\|^{2}_{L^2}\|\nabla\bar{u}\|^4_{L^2};
\end{aligned}\end{equation}
By duality, Sobolev embedding and Young's inequality, we have
\begin{align}
&\left|\int_{\R^3}\delta\!\rho\dot{\bar{u}}\cdot\delta\!u\,\dx\right|\leq \|\delta\!\rho\|_{\dot{W}^{-1,3}}\|\dot{\bar{u}}\cdot\delta\!u\|_{\dot{W}^{1,3/2}}\nonumber\\
 &\leq \|\delta\!\rho\|_{\dot{W}^{-1,3}}(\|\nabla\dot{\bar{u}}\|_{L^2}\|\delta\!u\|_{L^6}+\|\dot{\bar{u}}\|_{L^6}\|\nabla\delta\!u\|_{L^2})\leq C\|\delta\!\rho\|_{\dot{W}^{-1,3}}\|\nabla\dot{\bar{u}}\|_{L^2}\|\nabla\delta\!u\|_{L^2}\label{uniqueness.2}\\
 &\leq \f14 \|\nabla\delta\!u\|^{2}_{L^2}+C t^{-\f32}\|\delta\!\rho\|^2_{\dot{W}^{-1,3}} \cdot t^{\f32}\|\nabla\dot{\bar{u}}\|^2_{L^2}\leq\f14 \|\nabla\delta\!u\|^{2}_{L^2}+C X^2(t)t^{\f32}\|\nabla\dot{\bar{u}}\|^2_{L^2}.\nonumber
 \end{align}
Hence plugging \eqref{uniqueness.1} and \eqref{uniqueness.2} into \eqref{Eq.delta_u_energy} gives 
\begin{align*}
    \frac{\mathrm d}{\dt}\|\sqrt\rho\delta\!u\|_{L^2}^2+\|\nabla \delta\!u(t)\|^2_{L^2}\leq C\|\sqrt\rho\delta\!u(t)\|_{L^2}^2\|\nabla \bar u(t)\|_{L^2}^4+CX^2(t)t^{3/2}\|\nabla\dot{\bar{u}}(t)\|^2_{L^2}.
\end{align*}
By Gr\"onwall's lemma, $\nabla \bar u\in L^4(0,T;L^2(\R^3))$ and $t^{3/4}\nabla \dot{\bar u}\in L^2(0,T;L^2(\R^3))$, we have
\begin{align*}
\|\sqrt\rho\delta\!u(t)\|_{L^2}^2+\int_0^t\|\nabla\delta\!u(s)\|_{L^2}^2\,\ds&\leq C\exp\left(\int_0^t\|\nabla \bar u(s)\|_{L^2}^4\,\ds\right)\int_0^tX^2(s)s^{3/2}\|\nabla\dot{\bar u}(s)\|_{L^2}^2\,\ds\\
&\leq C\int_0^tX^2(s)\g(s)\,\ds,
\end{align*}
where $0\leq\g(s):=s^{3/2}\|\nabla\dot{\bar u}(s)\|_{L^2}^2\in L^1(0, T)$. Thus,
\begin{equation}\label{Eq.Y<X}
    Y^2(t)\leq C\int_0^tX^2(s)\g(s)\,\ds.
\end{equation}
It follows from \eqref{Eq.X<Y} and \eqref{Eq.Y<X} that $X^2(t)\leq C\int_0^tX^2(s)\g(s)\,\ds$, hence by Gr\"onwall's lemma we have $X(t)\equiv0$, then by \eqref{Eq.Y<X} we know that $Y(t)\equiv0$. This completes the proof of Theorem \ref{Thm.uniqueness}.
\end{proof}

\subsection{Proof of the uniqueness part of Theorem \ref{Thm.existence}}
Let $d=3$. Let $(\rho, u)$ and $(\bar \rho, \bar u)$ be two solutions to \eqref{INS} satisfying all the properties listed in Theorem \ref{Thm.existence} with the same initial data $(\rho_0, u_0)$. To prove $(\rho, u)=(\bar \rho, \bar u)$, it suffices to check that Theorem \ref{Thm.uniqueness} is applicable. Denote $\delta\!\rho:=\rho-\bar\rho$ and $\delta\!u:=u-\bar u$. Let $T>0$. 

Firstly, by \eqref{Eq.rho_W-13} we know that 
$$\|\rho(t)-\rho_0\|_{\dot{W}^{-1,3}}+\|\bar\rho(t)-\rho_0\|_{\dot{W}^{-1,3}}\leq Ct\|u_0\|_{\dot{H}^{\f12}},$$
thus $t^{-\f34}\delta\!\rho$ belongs to $L^{\infty}(0,T;\dot{W}^{-1,3}(\R^3))$. Next we  show that $
\delta\!u$ is in the energy space. Note that
\begin{equation}\begin{aligned}\label{Eq.u-u_0_L^2}
\int_{\R^3} \rho(t,x)|u(t,x)-\bar u(t,x)|^2\,\dx=\int_{\R^3}&\rho_0(x)|u(t,x)-\bar u(t,x)|^2\,\dx\\
&+\int_{\R^3}(\rho(t,x)-\rho_0(x)) |u(t,x)-\bar u(t,x)|^2\,\dx.   
\end{aligned}\end{equation}
By \eqref{Eq.rho_0ut_est}, we have 
\begin{equation}\begin{aligned}\label{Eq.u-u_0_rho_0}
\|\sqrt{\rho_0}(u(t)-u_0)\|_{L^2(\R^3)}&\leq C\int_0^t \|\sqrt{\rho_0}u_t(s)\|_{L^2}\,\ds=\int_0^t s^{-\f14}\|s^{\f14}\sqrt{\rho_0}u_t(s)\|_{L^2}\,\ds\\
&\leq Ct^{\f14}\|s^{\f14}\sqrt{\rho_0}u_t\|_{L^2(0,t;L^2)}\leq Ct^{\f14}\|u_0\|_{\dot{H}^{\f12}},
\end{aligned}\end{equation}
and the same inequality holds for $\bar u$. On the other hand, by duality, Sobolev embedding, \eqref{Eq.rho_W-13} and \eqref{Eq.low_order_est_rhou}, we have
\begin{align}
&\Bigl|\int_{\R^3}(\rho(t)-\rho_0) |u(t)-\bar u(t)|^2\,\dx\Bigr| \leq \|\rho(t)-\rho_0\|_{\dot{W}^{-1,3}}\|\nabla (|u(t)-\bar u(t)|^2)\|_{L^{\f32}}\nonumber\\
&\leq 2\|\rho(t)-\rho_0\|_{\dot{W}^{-1,3}}\|\nabla u(t)-\nabla\bar u(t)\|_{L^{2}}\|u(t)-\bar u(t)\|_{L^{6}} \nonumber\\
&\leq Ct\|\nabla u(t)-\nabla \bar u(t)\|^2_{L^{2}}\|u_0\|_{\dot{H}^{\f12}}\nonumber\\
&\leq Ct^{\f12} \|t^{1/4}(\nabla u,\nabla\bar u)\|^2_{L^{\infty}(\R^+;L^{2})}\|u_0\|_{\dot{H}^{\f12}}
\leq Ct^{\f12}\|u_0\|^3_{\dot{H}^{\f12}}.\label{Eq.u-u_0_duality}
\end{align}
Plugging \eqref{Eq.u-u_0_rho_0} and \eqref{Eq.u-u_0_duality} into \eqref{Eq.u-u_0_L^2} gives that $\sqrt\rho(u-\bar u)\in L^{\infty}(0,T;L^2(\R^3))$. By \eqref{Eq.nabla_u_L4L2}, we have $(\nabla u,\nabla\bar u)\in L^4(0,T;L^2(\R^3))\subset L^2(0,T;L^2(\R^3))$.
From this, we eventually conclude that $\sqrt\rho\delta\!u\in L^{\infty}(0,T;L^2(\R^3))$ and $\nabla\delta\!u\in L^2(0,T;L^2(\R^3))$. Finally, by Theorem \ref{Thm.existence}, we have $t^{\f12}\nabla \bar u\in L^2(0,T;L^{\infty}(\R^3))$ and $t^{\f34}\nabla\dot{\bar u}\in L^2(0,T;L^2(\R^3))$. Hence, all the assumptions of Theorem \ref{Thm.uniqueness} are satisfied by the solutions $(\rho, u)$ and $(\bar \rho, \bar u)$ constructed in Theorem \ref{Thm.existence}, which are thus equal to each other.

\section* {Acknowledgments}

D. Wei is partially supported by the National Key R\&D Program of China under the grant 2021YFA1001500. Z. Zhang is partially supported by NSF of China under Grant 12288101.

\end{document}